\documentclass[11pt]{amsart}
\usepackage{amssymb,amsfonts,amsxtra,amscd,mathrsfs,amsthm,stmaryrd}
\usepackage[all]{xy}
\usepackage{fullpage}
\usepackage{graphicx}
\theoremstyle{plain}
\newtheorem{theorem}{Theorem}[section]
\newtheorem{lemma}[theorem]{Lemma}

\newtheorem{prop}[theorem]{Proposition}
\theoremstyle{definition}
\newtheorem{defi}[theorem]{Definition}
\newtheorem{example}[theorem]{Example}
\theoremstyle{remark}
\newtheorem{rem}[theorem]{Remark}
\numberwithin{equation}{section}
\newcommand{\ai}{\ensuremath{A_{\infty}}}
\newcommand{\li}{\ensuremath{L_{\infty}}}
\newcommand{\gf}{\ensuremath{\mathbb{K}}}
\newcommand{\Mor}{\ensuremath{\mathcal{M}}}
\newcommand{\Res}{\ensuremath{\mathcal{R}}}
\newcommand{\NCBVGauss}[1]{\ensuremath{\mathrm{P}^{{\rm nc}}_{\hbar=1}\left[#1\right]}}
\newcommand{\ComBVGauss}[1]{\ensuremath{\mathrm{P}_{\hbar=1}\left[#1\right]}}
\newcommand{\NCHam}[1]{\ensuremath{\widehat{\mathrm{H}}\left[#1\right]}}
\newcommand{\NCHampoly}[1]{\ensuremath{\mathrm{H}\left[#1\right]}}
\newcommand{\NCHamP}[1]{\ensuremath{\widehat{\mathrm{H}}_+\left[#1\right]}}
\newcommand{\NCHamPpoly}[1]{\ensuremath{\mathrm{H}_+\left[#1\right]}}
\newcommand{\ComBV}[1]{\ensuremath{\widehat{\mathrm{P}}_{\hbar}\left[#1\right]}}
\newcommand{\ComBVpoly}[1]{\ensuremath{\mathrm{P}_\hbar\left[#1\right]}}
\newcommand{\NCBVpoly}[1]{\ensuremath{\mathrm{P}^{{\rm nc}}_\hbar\left[#1\right]}}
\newcommand{\NCBV}[1]{\ensuremath{\mathrm{P}^{\wedge{\rm nc}}_\hbar\left[#1\right]}}
\newcommand{\Pois}[1]{\ensuremath{\widehat{\mathrm{P}}\left[#1\right]}}
\newcommand{\Poispoly}[1]{\ensuremath{\mathrm{P}\left[#1\right]}}
\newcommand{\NonComBV}[1]{\ensuremath{\widehat{\mathrm{P}}^{{\rm nc}}_{\gamma,\nu}\left[#1\right]}}

\newcommand{\CycHoch}[1]{\ensuremath{CC\left(#1\right)}}
\newcommand{\ChEilp}[1]{\ensuremath{CE_{+}\left(#1\right)}}
\newcommand{\MCset}[1]{\ensuremath{\mathcal{MC}\left(#1\right)}}
\newcommand{\MCfiberset}[2]{\ensuremath{\mathcal{MC}|_{#2}\left(#1\right)}}
\newcommand{\MCfibermod}[2]{\ensuremath{\widetilde{\mathcal{MC}}|_{#2}\left(#1\right)}}
\newcommand{\Morph}[2]{\ensuremath{\mathsf{Mor}\left(#1,#2\right)}}
\newcommand{\MorphH}[2]{\ensuremath{\widetilde{\mathsf{Mor}}\left(#1,#2\right)}}
\newcommand{\Morphset}[3]{\ensuremath{\mathsf{Mor}_{#1}\left(#2,#3\right)}}
\newcommand{\Morphmod}[3]{\ensuremath{\widetilde{\mathsf{Mor}}_{#1}\left(#2,#3\right)}}
\newcommand{\Morphfiberset}[4]{\ensuremath{\mathsf{Mor}^{|#4}_{#1}\left(#2,#3\right)}}
\newcommand{\Morphfibermod}[4]{\ensuremath{\widetilde{\mathsf{Mor}}^{|#4}_{#1}\left(#2,#3\right)}}
\newcommand{\Obs}[1]{\ensuremath{\mathrm{Obs}\left(#1\right)}}
\newcommand{\gl}[2]{\ensuremath{\mathfrak{gl}_{#1}(#2)}}
\newcommand{\mat}[2]{\ensuremath{\mathrm{M}_{#1}(#2)}}
\newcommand{\her}[1]{\ensuremath{\mathfrak{h}_{#1}}}
\newcommand{\symg}[1]{\ensuremath{\mathbb{S}_{#1}}}
\newcommand{\cycg}[1]{\ensuremath{\mathbb{Z}/{#1}\mathbb{Z}}}
\newcommand{\twodim}{\ensuremath{\mathcal{A}}}
\newcommand{\innprod}{\ensuremath{\langle -,- \rangle}}
\newcommand{\invlim}[2]{\ensuremath{\varprojlim_{#1} #2}}
\newcommand{\dilim}[2]{\ensuremath{\varinjlim_{#1} #2}}
\newcommand{\cotimes}{\ensuremath{\widehat{\otimes}}}
\newcommand{\sym}{\ensuremath{S}}
\newcommand{\csym}{\ensuremath{\widehat{S}}}
\newcommand{\symp}{\ensuremath{S_+}}
\newcommand{\csymp}{\ensuremath{\widehat{S}_+}}
\newcommand{\ten}{\ensuremath{T}}
\newcommand{\cten}{\ensuremath{\widehat{T}}}

\newcommand{\ctenp}{\ensuremath{\widehat{T}_+}}
\newcommand{\clieh}{\ensuremath{\mathrm{C}_*}}
\newcommand{\cliec}{\ensuremath{\mathrm{C}^*}}
\newcommand{\hlieh}{\ensuremath{\mathrm{H}_*}}

\def\d{{\mathrm{d}}}
\def\tr{{\mathrm{Tr}}}
\def\lie{\mathrm{Lie}}
\def\xto{\xrightarrow}
\def\End{{\mathrm{End}}}
\def\Hoch{{\mathrm{Hoch}}}
\def\Cyc{{CC}}
\def\RR{\mathbb R}
\def\ZZ{\mathbb Z}
\def\cC{\mathcal C}
\DeclareMathOperator{\id}{id}

\begin{document}

\title{Large $N$ phenomena and quantization of the Loday-Quillen-Tsygan theorem}

\author{Gr\'egory Ginot}
\address{Universit\'e Sorbonne Paris Nord, Laboratoire de G\'eom\'etrie, Analyse et Applications, LAGA, CNRS, UMR 7539, F-93430, Villetaneuse, France}  \email{ginot@math.univ-paris13.fr}
\author{Owen Gwilliam}
\address{Department of Mathematics and Statistics, Lederle Graduate Research Tower 1623D, University of Massachusetts Amherst, 710 N. Pleasant Street, Amherst, MA 01003-9305. USA.} \email{gwilliam@math.umass.edu}
\author{Alastair Hamilton}
\address{Department of Mathematics and Statistics, Texas Tech University, Lubbock, TX 79409-1042. USA.} \email{alastair.hamilton@ttu.edu}
\author{Mahmoud Zeinalian}
\address{Department of Mathematics, Lehman College of CUNY, 250 Bedford Park Blvd W, Bronx, NY 10468. USA.} \email{mahmoud.zeinalian@lehman.cuny.edu}

\begin{abstract}
We offer a new approach to large $N$ limits using the Batalin-Vilkovisky formalism, both commutative and noncommutative,
and we exhibit how the Loday-Quillen-Tsygan Theorem admits BV quantizations in that setting.
Matrix integrals offer a key example: we demonstrate how this formalism leads to a recurrence relation that in principle allows us to compute all multi-point correlation functions. We also explain how the Harer-Zagier relations may be expressed in terms of this noncommutative geometry derived from the BV formalism.

As another application, we consider the problem of quantization in the large $N$ limit and demonstrate how the Loday-Quillen-Tsygan Theorem leads us to a solution in terms of noncommutative geometry. These constructions are relevant to open topological field theories and string field theory, providing a mechanism that relates moduli of categories of branes to moduli of brane gauge theories.
\end{abstract}
\thanks{This work was supported by the Simons Foundation by grant 279462 and by the National Science Foundation by grant 1812049.}

\dedicatory{To Albert Schwarz, with great admiration.}

\maketitle

\section{Introduction}

Many constructions involving $N\times N$ matrices exhibit surprising behavior in the limit as $N$ goes to infinity.
The archetypal example --- of deep and abiding interest to both physics and mathematics --- is the theory of random matrices {\it \`a la} Wigner,
but stable patterns in the large $N$ limit appear in many domains, ranging from $K$-theory to gauge theory.
Our central goal here is to offer a new view on a class of large $N$ phenomena by relating a homological result (the Loday-Quillen-Tsygan theorem) to quantum and probabilistic systems (such as Gaussian random matrices).
The mechanism by which these domains interact is the Batalin-Vilkovisky (BV) formalism,
which is a homological approach to formulating the path integral.
Although we end up constructing a general and abstract relationship, this introduction will focus on a fundamental example first and then on interpreting our results as an example of gauge-string duality.

Before jumping into details, however, a summary of our main ideas may  be helpful;
and we phrase these for two audiences.
From the perspective of physics, this paper shows that the BV formalism offers a novel approach to large $N$ gauge theory because it offers a mechanism for the emergence of stringy phenomena,
essentially by string field theory and the noncommutative BV formalism.
(We only examine zero-dimensional field theories in this paper; Movshev-Schwarz and Costello-Li offer higher-dimensional examples.)
From the perspective of homological algebra, this paper shows that the Loday-Quillen-Tsygan (LQT) theorem admits a natural quantization when applied to {\em cyclic} $\ai$-algebras.

In addition to studying BV quantizations of the LQT theorem,
we also analyze how the LQT map intertwines the deformation theory of a cyclic $\ai$-algebra $A$ with the deformation theory of the cyclic $\li$-algebras $\gl{N}{A}$.
From a physical viewpoint, this amounts to understanding how the moduli of noncommutative BV theories (a.k.a. topological string theories) controls the moduli of associated commutative BV theories (a.k.a. brane gauge theories).

There is enormous prior work on all these topics,
and it would be difficult (and possibly unhelpful) to survey a substantive fraction of the literature.
So far as we know, Movshev and Schwarz were the first to intertwine the BV formalism with cyclic homology in the style advocated by this paper.
See \cite{MS4} for a good starting point and survey of their beautiful and eye-opening body of work on supersymmetric gauge theories \cite{MS1, MS2, MS3}.
Schwarz also influenced us via \cite{SchACS}, where a relationship between open topological field theories and generalized Chern-Simons theories is explained.
More recently in \cite{CL1, CL2}, Costello and Li have used related ideas in examining how BCOV theory (a.k.a. Kodaira-Spencer gravity) relates to holomorphic Chern-Simons theory in the large $N$ limit.
We took inspiration from both of these predecessors, but our focus and context are rather different; closer in spirit (but exhibiting nontrivial differences) to Barannikov's lovely work \cite{Bar}.
Eventually we hope our methods may be fruitful in field-theoretic contexts as well.
Another approach to closely related phenomena, with an emphasis on moduli of surfaces, is developed by Cielebak, Fukaya, and Latschev, via $IBL_\infty$-algebras \cite{CFL};
we intend to give a comparison with our work in the future.
We also expect a close connection between our results and those of Berest, Ramadoss, and collaborators (see \cite{BeFeRa, BeRa} to start) but do not explore that here.

Note that in \cite{Iseppi, IsVS} there is a rather different approach to intertwining BV formalism, matrix models, and Connesian noncommutative geometry by Iseppi and van Suijlekom.

Throughout the paper, we have used the odd/even terminology to indicate $\mathbb{Z}/2$-graded objects.
In several discussions, this $\mathbb{Z}/2$-grading is a remnant of a stronger $\mathbb Z$-grading.
While a $\mathbb{Z}$-grading might be richer,
we have suppressed such subtleties in favor of continuity and ease.
Experienced readers desiring a $\mathbb{Z}$-grading will know how to provide the necessary details.

\subsection{What is the LQT theorem and what would it mean to quantize it?}

Let $\gf$ be a characteristic zero field and $A$ a unital associative algebra over $\gf$.
In fact, we can take $A$ to be a differential graded (dg) associative algebra or an $\ai$-algebra.
Then there is a sequence of maps
\[
A = \gl{1}{A} \hookrightarrow \gl{2}{A} \hookrightarrow \cdots \hookrightarrow \gl{N}{A} \hookrightarrow \cdots
\]
where each map extends an $N\times N$ matrix to an $(N+1)\times (N+1)$ matrix by putting zeros in the rightmost column and bottom row.
These are maps of Lie algebras (or if $A$ is dg, of dg Lie algebras; or if $A$ is $\ai$, of $\li$-algebras).
Let $\gl{\infty}{A}$ denote the Lie (respectively, dg Lie or $\li$) algebra of countably infinite square matrices with only finitely many nonzero entries, so that
\[
\gl{\infty}{A} = \dilim{N}{\gl{N}{A}},
\]
the colimit of this diagram.

Taking Lie algebra homology, one finds a sequence of maps of graded vector spaces
\[
\hlieh(\gl{1}{A}) \to \cdots \to \hlieh(\gl{N}{A}) \to \cdots \to \hlieh(\gl{\infty}{A})
\]
arising from a sequence of chain maps
\[
\clieh(\gl{1}{A}) \to \cdots \to \clieh(\gl{N}{A}) \to \cdots \to \clieh(\gl{\infty}{A}).
\]
What was found by Tsygan and (independently) Loday and Quillen, is that there is a simpler chain complex defined directly in terms of the algebra $A$ that encodes Lie algebra homology after stabilization.

\begin{theorem}
For any unital $\ai$-algebra $A$ over a field $\gf$ of characteristic zero, there is a natural quasi-isomorphism
\[
\clieh(\gl{\infty}{A}) \to \sym(\Cyc_*(A)[1])
\]
where $\Cyc_*(A)$ denotes the cyclic chain complex of $A$ and $\sym(-)$ denotes the free graded-commutative algebra.
This construction is functorial in~$A$.
\end{theorem}

We will review this construction in our paper (including the definition of cyclic homology), but it boils down to clever and concrete invariant theory calculations with matrices.
The graded linear dual map
\[
\csym(\Cyc^*(A)[-1]) \to \cliec(\gl{\infty}{A})
\]
relates (the completed symmetric algebra on) cyclic cohomology to the stable Lie algebra cohomology.
For us, much of the interest is in the induced maps
\[
\csym(\Cyc^*(A)[-1]) \to \cliec(\gl{N}{A})
\]
relating the cyclic cohomology to Lie algebra cohomology at some finite~$N$.

The reader should (rightly!) wonder how this theorem could be quantized and what that could mean.
For the moment, we remark that when doing classical field theory in the BV formalism,
the classical observables form a {\em dg commutative algebra} that often admits a nice description as $\cliec(\mathfrak{g})$  for some $\li$-algebra $\mathfrak{g}$.
A particularly pertinent example is Chern-Simons theory on an oriented 3-manifold $X$ for the Lie algebra $\mathfrak{u}_N$.
In that case, working with connections on the trivial principal $U_N$-bundle, we have $\mathfrak{g} = \Omega^*(X) \otimes \mathfrak{u}_N$,
the algebra of $\mathfrak{u}_N$-valued differential forms.
Of course, complexifying $\mathfrak{u}_N$ gives $\gl{N}{\mathbb{C}}$, which places us in a domain where the LQT theorem is relevant.
In particular,
\[
\cliec(\mathfrak{g})\otimes\mathbb{C} = \cliec\left(\gl{N}{\Omega^*_{\mathbb{C}}(X)}\right),
\]
where on the right-hand side we consider complex valued differential forms $\Omega^*_{\mathbb{C}}(X)$ and work over the ground field $\mathbb{C}$. In this way, complex-valued observables occur through complexifying the Lie algebra.
The LQT theorem then says that the large $N$ limit of classical Chern-Simons theory is encoded by $\csym(\Cyc^*(\Omega^*_{\mathbb{C}}(X))[-1])$.
Since the cyclic cohomology of the de Rham complex encodes the rotation-equivariant homology of the free loop space $LX$ (up to subtleties about the fundamental group),
the LQT theorem implies that the large $N$ limit of classical Chern-Simons theory on $X$ describes something about the free loop space $LX$, which is manifestly relevant to string theory and string topology.
(In the remarkable \cite{SchACS}, Schwarz lays out this relationship.)

It is then interesting to ask whether there is a \emph{quantum} version,
where quantization is also done within the BV formalism.
At the classical level, the observables are not only differential graded but actually form a {\em shifted Poisson dg algebra}.
In the case of Chern-Simons theory, the shifted Poisson bracket arises from the bilinear pairing on fields $\mathfrak{g}$ by combining the wedge product of forms with the trace pairing on matrices.
BV quantization amounts to deforming the differential on the classical observables in a way that depends on the shifted Poisson bracket on those classical observables.
In this sense it resembles deformation quantization, except that a differential is deformed rather than a multiplication.
Below in Section~\ref{sec: gauge-string} we revisit these ideas in a broader context, using branes and functorial field theory.

More abstractly, the classical BV formalism can be seen as the study of cyclic $\li$-algebras,
i.e. $\li$-algebras with an invariant pairing of odd degree.
(We give a precise and complete definition of these inside the paper.)
In this view, the underlying graded vector space of the $\li$-algebra encodes the field content and the brackets encode the homogeneous terms of the action functional.
The Chevalley-Eilenberg cochains of the $\li$-algebra are the classical observables,
and the cyclic structure determines the (typically) shifted Poisson bracket.

On the other hand, given a cyclic $\ai$-algebra $A$, one obtains an infinite family of $\li$-algebras $\gl{N}{A}$,
with a canonical invariant pairing that combines the trace pairing on matrices with the pairing on $A$.
Hence, cyclic $\ai$-algebras provide a source of many examples for the classical BV formalism.
In this paper we refine the LQT theorem by showing that for a \emph{cyclic} $\ai$-algebra $A$,
the LQT maps
\[ \csym(\Cyc^*(A)[-1]) \to \cliec(\gl{N}{A}) \]
are shifted Poisson and that they admit a natural BV quantization.
For gauge-type theories, these provide BV quantizations that are ``uniform in~$N$.''

We wish to note something striking about the LQT map: the domain describes a {\em free} theory in the sense that the differential is linear (and hence corresponds to a purely quadratic action functional).
In other words, there is a free theory governing the gauge theories in the large $N$ limit.
This feature explains the remarkable simplicities that emerge in this limit.
On the other hand, the quantization we identify --- which corresponds to the standard BV quantization on the Lie or gauge-theoretic side --- does not look like the ``obvious'' quantization of this free theory.
Instead the BV Laplacian (i.e., the deformation of the differential) has a fascinating structure,
related to ribbons.

Let us point out a drawback of our methods:
they do not apply directly to infinite-dimensional algebras, such as de Rham complexes, because we require strict non-degeneracy of the cyclic structure.
True gauge theories, in all their glory, lie outside the scope of this paper and would require a renormalized version of our methods.
Nonetheless, interesting examples exist, as we explain below when discussing branes.

\subsection{The LQT Theorem and the Gaussian unitary ensemble}

Let $\her{N}$ denote the vector space of Hermitian $N\times N$ matrices.
Equip it with the Gaussian measure
\[ \mu_N := \frac{1}{Z_N}e^{-\frac{1}{2}\tr(X^2)}\d X, \]
normalized so that
\[
\frac{1}{Z_N}\int_{\her{N}} e^{-\frac{1}{2}\tr(X^2)}\d X=1.
\]
In other words, we have identified a space of {\it random matrices};
in particular, physicists will recognize that it is a space of {\it random Hamiltonians} for quantum mechanical systems with finite-dimensional Hilbert spaces (without any time dependence).
A wonderful surprise, discovered by Wigner, is the emergence of simple patterns as $N$ gets large.
One goal of this paper is to offer an explanation for some of these patterns via the Loday-Quillen-Tsygan theorem.

As the measure is naturally invariant under conjugation by unitary transformations,
it is natural to ask first about invariant functions of these random matrices. Define,
\[ I^N_k := \int_{\her{N}} \tr(X^k) \d \mu_N \]
to be the {\em expected value of the single-trace operator} $\tr(X^k)$,
which is a conjugation-invariant moment of the Gaussian measure $\mu_N$.

Recall the case $N=1$, which is simply the Gaussian measure on the real line.
The odd moments vanish as the measure is an even function and an odd power is an odd function.
Thus, only the even moments are non trivial.
A classic computation shows,
\begin{align*}
I^1_{2k} &=\int_\RR x^{2k} e^{-\frac{1}{2} x^2}\d x=\frac{(2k)!}{2^k k!},\\
&=(2k-1)!! := (2k-1)(2k-3)(2k-5) \cdots 5 \cdot 3 \cdot 1;
\end{align*}
a result sometimes known as Wick's lemma.
These moments satisfy a simple recurrence relation,
\[ I^1_{2k}=(2k-1)I^1_{2k-2}. \]

For general $N$, Harer and Zagier obtained a remarkable closed formula for~$I^N_{2k}$.

\begin{theorem}[Harer-Zagier]
\begin{equation} \label{eqn_HZformula}
I^N_{2k} = \frac{(2k)!}{2^k k!} \sum\limits_{m=0}^k 2^m {k \choose m} {N \choose m+1}
\end{equation}
\end{theorem}

Some patterns emerge upon careful inspection.
For instance, the leading coefficient of the polynomial $I^N_{2k}$ is the integer $\frac{(2k)!}{k!(k+1)!}$, known as the $k$th Catalan number, which counts the number of possible ways of gluing the edges of a $2k$-gon to form a sphere.
Similarly, the coefficient of the next nonzero term (which has degree $k-1$) is the number of ways that gluing the edges of a $2k$-gon results in a torus.

There is also a useful recurrence relation for $k\geq 2$,
\begin{equation} \label{eqn_HZrec}
I^N_{2k}=\frac{(4k-2)}{(k+1)}NI^N_{2k-2}+\frac{(k-1)(2k-1)(2k-3)}{(k+1)}I^N_{2k-4},
\end{equation}
fully determined by the initial values $I^N_0=N$ and~$I^N_2=N^2$.
One can verify this recurrence relation by elementary but nontrivial analysis using Hermite polynomials and their properties (see, e.g., \cite{Ledoux}),
and thence deduce the closed formula~\eqref{eqn_HZformula}.

A well-known consequence of the general formula \eqref{eqn_HZformula} is Wigner's celebrated semi-circle law,
which we formulate in a limited but simple form.

\begin{theorem}[Wigner]
For any polynomial $f \in \RR[x]$,
\[
\lim_{N \to \infty} \frac{1}{N}\int_{\her{N}} \tr \,\left(f\left( X/\sqrt{N}\right)\right) \d\mu_N = \frac{1}{2\pi}\int_{-2}^2 f(x)\sqrt{4-x^2}\d x.
\]
\end{theorem}

In other words, the expected value on matrices --- \emph{as we deal with arbitrarily large matrices} --- is given by integrating against the height of a circle with radius~2.

\begin{proof}[Proof of Wigner's law as a consequence of Harer-Zagier recurrence]
It suffices to prove the theorem for each $f(x)=x^{2n}$.
The right hand side is a direct computation, and it recovers precisely the Catalan number that was found for the left hand side.
\end{proof}

Our paper provides a {\em homological} approach to studying Gaussian random matrices,
particularly how to compute the expected values of conjugation-invariant observables.
When we develop this example, we will see how the large $N$ limit of our approach encodes the Harer-Zagier result and hence Wigner's law.

The first step is to rephrase the case of finite $N$ homologically.
There is a standard way to do this in the BV formalism \cite{GJF},
but we sketch the argument here.
A traditional method for encoding integration homologically is the de Rham complex.
For a Gaussian measure $\mu = e^{-Q} \d V$ on a real vector space $V$,
one can show that inside $\Omega^*_{\mathbb{C}}(V)$, there is a subcomplex
\[
\Omega^*_{\mathbb{C},\mu}(V) = \{ p e^{-Q} \d x_1 \wedge \cdots \wedge \d x_k \,:\, p \text{ polynomial}, x_i \text{ linear} \}
\]
consisting of polynomial de Rham forms multiplied by $e^{-Q}$.
Since polynomials have finite moments when integrated against a Gaussian measure,
this subcomplex has one-dimensional cohomology concentrated in the top degree.
In particular, one can take $V$ to be the Hermitian $N\times N$ matrices with measure~$\mu_N$.

We now state a key result which may be derived directly from the machinery of the BV formalism: \emph{a BV quantization of $\cliec(\gl{N}{\twodim})$, where $\twodim$ is described below, encodes Gaussian integration over Hermitian matrices}.
An explicit description appears in Section~\ref{sec_random}, but we sketch it here.

Let $\twodim$ denote the cochain complex $\mathbb{C} \xto{id} \mathbb{C}$ concentrated in degrees 1 and 2,
viewed as a dg associative algebra with \emph{trivial} multiplication.
Then for every natural number $N$, there is an isomorphism of cochain complexes
\[
\Omega^*_{\mathbb{C},\mu_N}(\her{N})[-N^2] \simeq (\sym(\gl{N}{\twodim}^*[-1]), \d_{\lie} + \Delta_N),
\]
where $\d_{\lie}$ denotes the differential on Chevalley-Eilenberg cochains $\cliec(\gl{N}{\twodim})$ and $\Delta_N$ denotes a deformation of the differential.
There is a lot to unwind in that statement,
but the important point is that this polynomial de Rham complex is a deformation of a Lie algebra cochain complex after shifting the de Rham complex into degrees $[-N^2,0]$.

These isomorphisms suggest that this simple-minded algebra $\twodim$ might play a universal role in studying the Gaussian Unitary Ensemble (GUE).
Let us plug it into the LQT theorem and see what happens.

Thankfully, the cyclic cochain complex of $\twodim$ is easy to describe.
One can identify the degree 0 component of $\Cyc^*(\twodim)[-1]$ with the maximal ideal $\langle x \rangle$ inside the polynomial ring $\mathbb{C}[x]$.
In degree 0 the Lie algebra cochains are $\sym(\gl{N}{\mathbb{C}}^*)$.
The LQT map is also easy to describe in degree 0, and it is
\[
\begin{array}{ccc}
\langle x \rangle & \to & \sym(\gl{N}{\mathbb{C}}^*)\\
x^n & \mapsto & \tr(X^n).
\end{array}
\]
That is, the LQT map produces the single-trace operators.
When we extend our consideration to the entire domain $\sym(\Cyc^*(\twodim)[-1])$ of the LQT map,
it produces the multi-trace operators.
At the classical level, the LQT theorem focuses on the observables relevant to random matrix theory.

Our quantized LQT theorem then produces a BV quantization on the cyclic side that maps to the BV quantization of $\cliec(\gl{N}{\twodim})$ encoding Gaussian integration.
On the cyclic side, the Harer-Zagier recurrence appears as a relation among cocycles,
and the LQT map shows how that determines relationships among the expected values of tracial moments.
In this sense, our BV approach to the large $N$ limit captures Wigner's law.
Because the relationship is very explicit, we can also use this to explore interesting multi-trace recurrences.

We remark that it is natural to consider close variants of this story, where we replace Hermitian matrices by orthogonal or symplectic matrices.
(In probability theory these are known as the GOE and GSE.)
There is an analogous variant of the cyclic cohomology known as {\em dihedral cohomology}.
Loday and Procesi \cite{LodPro} proved an analog of the LQT theorem, relating dihedral cohomology of an involutive algebra $A$ to the large $N$ limit of Lie algebra cohomology for orthogonal or symplectic matrices with values in $A$.
Our techniques should be immediately applicable to this setting, so that we obtain analogs of classic results (e.g., expected values of multitrace operators) for GOE and GSE.

\subsection{A version of gauge-string duality}
\label{sec: gauge-string}

Our results apply naturally to topological string theories,
as encoded via the extended 2-dimensional topological field theories known as open-closed TFTs or as topological conformal field theories (TCFTs).
We will mention physical terminology but focus on mathematical statements.

Recall \cite{CosTCFT, LurieTFT} that a Calabi-Yau $A_\infty$-category $\cC$ determines a TCFT $Z_\cC$,
which we view as encoding a topological string theory.
This category captures the open sector of the topological string theory:
one interprets $\cC$ as the {\em $D$-branes} for the string theory, and
we will suggestively call an object $B$ of $\cC$ a brane.
What the TCFT assigns to the circle $Z_\cC(S^1)$ captures the closed sector of the string theory.
The fundamental theorem of TCFTs says that given $\cC$, the canonical value is $Z_\cC(S^1) \simeq \Hoch_*(\cC)$,
the Hochschild chains of the category.
The state space of the closed string field theory is, moreover, the cyclic chains~$\Cyc_*(\cC)$.
In \cite{CostGW}, Costello shows how the TCFT determines canonically a closed string field theory.
(By that we mean a solution to the quantum master equation in a BV algebra arising from the cyclic chains;
Costello relates it to the Sen-Zwiebach algebra.
For us, a closed string field theory need not be a Lagrangian field theory on some manifold,
such as BCOV theory.)

We now relate that story to ours. First, a cyclic $\ai$-algebra is a Calabi-Yau $\ai$-category with a single object.
Conversely, for any brane $B$ from a Calabi-Yau category $\cC$, the endomorphisms $\End_\cC(B)$ are modeled by a cyclic $\ai$-algebra.
To be more accurate, this $A_\infty$-algebra has a nondegenerate pairing at the level of cohomology,
and if $\End_\cC(B)$ has finite-dimensional cohomology, then one can find a representative $A_\infty$-algebra with a nondegenerate pairing on the cochain-level, by homotopy transfer.
(Often, however, the natural presentation of a brane yields an infinite-dimensional $\ai$-algebra,
an issue that we discuss below in the context of examples.)
In this sense, if one studies the TCFT associated to a single brane, one is in the setting of our results.

Now suppose one takes $N$ coincident copies of the same brane, and let's suppose that can be modeled by an $N$-fold direct sum $B^{\oplus N}$.
Then
\[
\End_\cC(B^{\oplus N}) \cong \gl{N}{\End_\cC(B)},
\]
namely matrices with values in endomorphisms of $B$.
Moreover, we can view $\gl{N}{\End_\cC(B)}$ as a cyclic $\li$-algebra and hence as presenting a classical BV theory.
We interpret it as the {\em brane gauge theory} arising from $N$ coincident copies of the brane~$B$.

\def\cO{\mathcal{O}}
\def\cF{\mathcal{F}}

Let us mention briefly two examples that might orient the reader.
\begin{itemize}
\item
Let $\cC$ denote an $A_\infty$-category enhancing the derived category of coherent sheaves on a projective Calabi--Yau manifold $X$. Given an object $\cF$ (e.g., a coherent sheaf), the derived endomorphisms are $\mathbb{R} \End_{\cO_X}(\cF)$.
(At the level of cohomology, it is the self-exts of $\cF$. One can find a convenient cochain representative by taking the Dolbeault complex of a complex of holomorphic vector bundles resolving $\cF$.)
As $X$ is projective, this $\ai$-algebra has finite-dimensional cohomology, and so after transferring the $\ai$-structure to the cohomology, we have a cyclic $\ai$-algebra, as desired.
As a concrete example, take $\cF = \cO_X$, in which case the Hochschild cochains are modeled by polyvector fields on $X$, and the cyclic cochains admit a model built from polyvector fields.
Pursuing this example leads naturally to the BCOV theory as studied in the BV formalism by Costello and Li \cite{CL1}.
Taking $X$ to be a Calabi--Yau 3-fold, the associated brane gauge theories are holomorphic Chern--Simons theories on~$X$.
\item
Let $\cC$ denote an $A_\infty$-category encoding $\infty$-local systems on a closed oriented smooth manifold $X$.
Given such a local system, its derived endomorphisms form an $\ai$-algebra over the de Rham complex of $X$.
As $X$ is closed, this $\ai$-algebra has finite-dimensional cohomology, and so after transferring the $\ai$-structure to the cohomology, we have a cyclic $\ai$-algebra, as desired.
As a concrete example, take the trivial local system given by the constant sheaf.
Its derived endomorphisms are modeled by the de Rham complex itself.
If $X$ is a 3-manifold, the associated brane gauge theories are precisely the perturbative Chern--Simons theories already discussed. Note that Schwarz described this picture in his remarkable paper~\cite{SchACS}.
\end{itemize}
There are many variants on these examples.

The Loday-Quillen-Tsygan theorem tells us that there is a canonical relationship between the large $N$ limit of the {\em classical} brane gauge theory for $B$ and the state space of the closed string theory determined by the brane $B$, as encoded in the cyclic cohomology of $\End_\cC(B)$.
This map can be seen as relating the moduli of closed string theory to the moduli of the brane gauge theory;
we examine this aspect in detail in Sections \ref{sec_obsthy} and \ref{sec_largeN},
which focus on deformation theory of commutative and noncommutative BV theories.
Alternatively, this map can be seen as relating classical observables for the closed string theory and the brane gauge theories;
our other results can be seen provide a quantization of that relationship.

Note that this story also relates to the TCFT determined by the whole category $\cC$.
A choice of brane $B$ can be viewed as choosing an $\ai$-functor from the single-object category with morphisms $\End_\cC(B)$ into $\cC$.
That determines a cochain map $\Cyc^*(\cC) \to \Cyc^*(\End_\cC(B))$.
It is natural to ask if the closed string field theory arising from the TCFT $Z_\cC$ maps to the quantization we produce for the large $N$ limit of the brane gauge theories.
We hope that our quantization receives a canonical map from that closed string field theory,
and that our quantization is equivalent to that produced by the TCFT arising from~$B$.
It is highly nontrivial to construct these functorial quantizations (cf. \cite{CalCosTu, CalTu}),
so we do not attempt to test these aspirations in the current paper.

We remark that our methods apply to configurations of branes and not just isolated branes.
If one takes a collection of branes, one can consider the subcategory they generate,
which (in good cases) admits a finite set of generators.
Hence this subcategory is Morita equivalent to an $\ai$-algebra,
and so our techniques apply (so long as the cohomology is finite).

\subsection{Layout of the paper}

In Section \ref{sec_homalgebra} we recall the basic homological algebra that provides the underpinnings of our approach to large $N$ phenomena. Here we recall the definitions of cyclic $\ai$ and $\li$-algebras and their cohomology. We also review a number of fundamental results, in the context of infinity-algebras; such as Morita invariance, the Loday-Quillen-Tsygan Theorem and the invariance of Chevalley-Eilenberg cohomology under the action of a semisimple Lie algebra.

Section \ref{sec_BVNonComGeom} reviews the basic framework of the Batalin-Vilkovisky formalism. This includes both the familiar commutative framework and the noncommutative framework that follows the style of Kontsevich. Here we explain how cyclic $\ai$ and $\li$-structures, along with their cohomology, may be encoded within this framework.

In Section \ref{sec_MinvLQTBV} we continue our description of the noncommutative BV formalism. We then proceed to explain the connection between the material of the first two sections. In particular, we describe how the LQT map intertwines the noncommutative BV formalism with its commutative counterpart.

Section \ref{sec_random} discusses the application of the preceding results to random matrix theory. Here we largely limit ourselves to explaining the role of noncommutative geometry and the Loday-Quillen-Tsygan Theorem in describing the $k$-point correlation functions. It is possible to use this framework to describe the large $N$ asymptotic behavior of these correlation functions; but to keep the current paper within reasonable limits, we intend to pursue this subject elsewhere.

In Section \ref{sec_obsthy} we introduce a general framework in which to discuss quantization in the Batalin-Vilkovisky formalism which applies equally well to both the commutative and noncommutative incarnations presented in this paper. We identify, both abstractly and for our examples, the cohomology theories controlling this process. Here we prove an appropriate analogue of the inverse function theorem: a map that induces an isomorphism on cohomology yields a bijection between moduli spaces of quantizations. We apply this theorem to both commutative and noncommutative examples.

Finally, in Section \ref{sec_largeN} we explain how this general framework applies to quantization in the large $N$ limit. Here it is necessary to have a quantization at each rank $N$ of the theory. We explain that, crudely speaking, the Loday-Quillen-Tsygan Theorem tells us that we must build this quantization using the noncommutative BV formalism.

\subsection{Acknowledgements}

We are grateful to a number of collaborators and colleagues who have spoken with us as we pursued these ideas,
notably Ga\"etan Borot, Dmitri Pavlov, and Kevin Costello.
We are grateful for the convivial atmosphere and financial support that MPIM supplied.
Subsequently, OG was lucky to draw AH into collaboration on this topic thanks to support from the Simons Foundation.
Surya Raghavendran and Philsang Yoo gave helpful feedback on a draft of this paper,
which clarified several points.
During work on this project, the National Science Foundation supported OG through DMS Grant No. 1812049.
Any opinions, findings, and conclusions or recommendations expressed in this material are those of the authors and do not necessarily reflect the views of the National Science Foundation.

\subsection{Notation and conventions}

Throughout the paper we work over a field $\gf$ of characteristic zero. We denote the algebra of $N\times N$ matrices with entries in $\gf$ by $\mat{N}{\gf}$.

Our convention will be to work with cohomologically graded objects, hence the suspension $\Sigma V$ of a graded vector space $V$ is defined by $\Sigma V^i := V^{i+1}.$
In fact, for much of the paper, we will generally work just with a $\ZZ/2$-grading (i.e., with super vector spaces),
leaving it to the reader to lift to a $\ZZ$-grading, as interested.
(Our results hold at that level but one must then track the degrees of, say, the odd pairings and give $\hbar$ a cohomological degree.)
When it is informative,
such as in Section~\ref{sec_random},
we will explicitly describe the lift.

A profinite vector space $V$ is an inverse limit
\[ V=\invlim{\alpha\in\mathcal{I}}{V_{\alpha}} \]
of finite-dimensional vector spaces $V_{\alpha}$. Such a presentation induces upon $V$ the inverse limit topology. A morphism of profinite vector spaces is a \emph{continuous} linear map. Since any vector space is the direct limit of its finite-dimensional subspaces, the category of ordinary vector spaces is anti-equivalent to the category of profinite vector spaces under the functor that takes a vector space to its linear dual. Since in this paper it is necessary for us to work with \emph{co}homology, profinite vector spaces and their associated constructions will accordingly make their appearance.

We define the completed tensor product of two profinite vector spaces $U$ and $V$ by
\[ U \cotimes V := \invlim{\alpha,\beta}{U_{\alpha}\otimes V_{\beta}}. \]
This corresponds to the usual tensor product under the above equivalence of categories.

We denote the symmetric group by $\symg{n}$. Given a profinite vector space $V$ we define the \emph{completed} tensor and symmetric algebras by
\[ \cten V:=\prod_{n=0}^{\infty}V^{\cotimes n} \quad\text{and}\quad \csym V:=\prod_{n=0}^{\infty} \big{[}V^{\cotimes n}\big{]}_{\symg{n}}; \]
where we make use of the convention that denotes coinvariants by a subscript and invariants by a superscript. We will also use the notation
\[ \ctenp V:=\prod_{n=1}^{\infty}V^{\cotimes n} \quad\text{and}\quad \csymp V:=\prod_{n=1}^{\infty} \big{[}V^{\cotimes n}\big{]}_{\symg{n}}. \]
We note that the standard symmetric and tensor algebras are defined in the usual way:
\[ \ten V:=\bigoplus_{n=0}^{\infty}V^{\otimes n} \quad\text{and}\quad \sym V:=\bigoplus_{n=0}^{\infty} \big{[}V^{\otimes n}\big{]}_{\symg{n}}. \]

\section{$\ai$-algebras, Morita invariance and the Loday-Quillen-Tsygan Theorem} \label{sec_homalgebra}

In this section we collect some of the basic facts about $\ai$ and $\li$-algebras that will underlie our approach to quantization in the large $N$ limit. Here we formulate, in the context of homotopy algebras, some of the essential results concerning the cohomology of these algebras. This includes the Morita invariance of Hochschild cohomology and the Loday-Quillen-Tsygan Theorem on the $N$-stable cohomology of $\gl{N}{A}$. These results will provide the main tools for our work in the later sections.

\subsection{$\ai$-algebras and cyclic cohomology}

We begin by recalling the definition of an $\ai$-algebra $A$ and its cyclic cohomology.

\begin{defi} \label{def_ainfinity}
An $\ai$-algebra is a vector space $A$ together with a continuous derivation,
\begin{equation} \label{eqn_ainfinityderivation}
m:\cten\Sigma A^*\to\cten\Sigma A^*
\end{equation}
of degree one such that $m^2=0$. Furthermore, this formal vector field should vanish at zero in the sense that its image is contained in $\ctenp\Sigma A^*$.
\end{defi}

Under the identification of $\cten\Sigma A^*$ with the dual $(\ten\Sigma A)^*$ of the tensor algebra on $\Sigma A$, an $\ai$-structure $m$ on a vector space $A$ is dual to a system of structure maps;
\begin{equation} \label{eqn_ainfintystructuremaps}
m_k: A^{\otimes k} \to A, \quad k\geq 1.
\end{equation}
The equation $m^2=0$ encodes the homotopy coherence relations.

\begin{defi} \label{def_ainfunital}
We say that an $\ai$-structure $m$ on a vector space $A$ is \emph{unital} if there is an element $1\in A$ such that:
\begin{itemize}
\item
For all $a\in A$, $m_2(1,a) = a = m_2(a,1)$.
\item
For all $k\neq 2$, the structure map $m_k(a_1,\ldots,a_k)$ vanishes whenever any argument $a_i$ is equal to $1$.
\end{itemize}
\end{defi}

The next step is to explain how matrices inherit an $\ai$-structure. Given an $\ai$-algebra $A$ consider the space,
\begin{equation} \label{eqn_matrixtensorrepresentation}
\mat{N}{A}=A\otimes\mat{N}{\gf}
\end{equation}
of $N\times N$ matrices with entries in $A$. This has the natural structure of an $\ai$-algebra which we now describe.

\begin{defi} \label{def_ainfmatrices}
Given an $\ai$-algebra $A$, the $\ai$-structure on $\mat{N}{A}$ is defined by tensoring the structure maps \eqref{eqn_ainfintystructuremaps} with the maps;
\[ \mat{N}{\gf}^{\otimes k} \to \mat{N}{\gf}, \qquad X_1\otimes X_2\otimes\cdots\otimes X_k \mapsto X_1 X_2\cdots X_k. \]
\end{defi}

One can easily check that this defines a new $\ai$-structure on $\mat{N}{A}$ which generalizes the usual associative algebra structure on matrices that occurs when $A$ is a strictly associative algebra.

Next we recall the definition of the cyclic cohomology of an $\ai$-algebra. This is the cohomology theory that will ultimately emerge in the $N$-stable limit to control the quantization process and the calculation of expectation values within the Batalin-Vilkovisky formalism.

\begin{defi}
Given an $\ai$-algebra $A$, the $\ai$-structure \eqref{eqn_ainfinityderivation} induces a differential on
\[ \CycHoch{A}:= \prod_{k=1}^\infty \big[(\Sigma A^*)^{\cotimes k}\big]_{\cycg{k}}. \]
The cohomology of the complex $\CycHoch{A}$ is called the \emph{cyclic cohomology} of $A$.
\end{defi}

Now we discuss the notion of a \emph{cyclic} $\ai$-algebra, which requires that we have a symmetric nondegenerate bilinear form $\innprod$ on $A$. In particular, $A$ must be finite-dimensional.

\begin{defi} \label{def_ainfinitycyclic}
We say that an $\ai$-structure $m$ on $A$ is \emph{cyclic} if the multilinear maps
\begin{equation} \label{eqn_cyclicainfstructuremaps}
a_0,\ldots,a_k \mapsto \langle m_k(a_0,\ldots,a_{k-1}),a_k \rangle, \quad k\geq 1
\end{equation}
coming from the structure maps \eqref{eqn_ainfintystructuremaps} are cyclically antisymmetric.
\end{defi}

We note that if $A$ is a cyclic $\ai$-algebra then so is $\mat{N}{A}$, where we combine the bilinear form on $A$ with the trace pairing on matrices.

\subsection{$\li$-algebras and Chevalley-Eilenberg cohomology} \label{sec_LinfCE}

Here we recall the definition of an $\li$-algebra and its Chevalley-Eilenberg cohomology.

\begin{defi}
An $\li$-algebra consists of a vector space $\mathfrak{g}$ and a continuous derivation,
\[ l:\csym\Sigma\mathfrak{g}^*\to\csym\Sigma\mathfrak{g}^* \]
of degree one satisfying $l^2=0$. As in Definition \ref{def_ainfinity} it must also vanish at the origin, meaning that its image is contained in $\csymp\Sigma\mathfrak{g}^*$. The cohomology of the complex $\csym\Sigma\mathfrak{g}^*$ with respect to $l$ is called the \emph{Chevalley-Eilenberg cohomology} of $\mathfrak{g}$ (with trivial coefficients). Since it splits off as a summand; we will, by an abuse of terminology, also refer to the subcomplex
\[ \ChEilp{\mathfrak{g}}:=\csymp\Sigma\mathfrak{g}^* \]
as the Chevalley-Eilenberg cochain complex, although it is more properly referred to as the \emph{reduced} Chevalley-Eilenberg cochain complex.
\end{defi}

As before, an $\li$-structure $l$ is dual to a system of antisymmetric structure maps;
\[ l_k:\mathfrak{g}^{\otimes k}\to\mathfrak{g}, \quad k\geq 1. \]
The equation $l^2=0$ imposes the homotopy Jacobi identities.

\begin{defi}
If $\mathfrak{g}$ is endowed with a symmetric nondegenerate bilinear form $\innprod$, then we say that an $\li$-structure $l$ on $\mathfrak{g}$ is cyclic if the multilinear maps
\begin{equation} \label{eqn_cycliclinfstructuremaps}
 y_0,\ldots, y_k \mapsto \langle l_k(y_0,\ldots,y_{k-1}),y_k \rangle, \quad k\geq 1
\end{equation}
are antisymmetric.
\end{defi}

The $\li$-structures that we consider in this paper come from commutator algebras, which we now describe.

\begin{defi} \label{def_commutatoralgebra}
Any continuous derivation on the completed tensor algebra naturally induces such a derivation on the completed symmetric algebra. Given an $\ai$-structure on $A$, we refer to the $\li$-structure induced on $A$ in this way as the \emph{commutator} $\li$-structure. In particular, we will denote by $\gl{N}{A}$ the commutator $\li$-algebra formed by $N\times N$ matrices with entries in $A$ that comes from Definition \ref{def_ainfmatrices}.
\end{defi}

We note that one immediate consequence of the above definition is that there is a map of complexes
\begin{equation} \label{eqn_cyctoCE}
\CycHoch{\mat{N}{A}}\to\ChEilp{\gl{N}{A}}.
\end{equation}

Now we consider the action of $\gl{N}{\gf}$ on the Chevalley-Eilenberg complex. The Lie algebra $\gl{N}{\gf}$ acts on $\gl{N}{A}$ in the standard way on the righthand factor of \eqref{eqn_matrixtensorrepresentation}. On $\csym\Sigma\gl{N}{A}^*$ it acts by derivations of this algebra. This action commutes with the commutator $\li$-structure $l$, since $\gl{N}{\gf}$ acts by derivations on the algebra $\mat{N}{\gf}$. Consequently, the $\gl{N}{\gf}$-invariants form a subcomplex of the Chevalley-Eilenberg complex. In fact, if $A$ is unital then this action is nullhomotopic; given a matrix $B$ in $\gl{N}{\gf}$, we have
\[ B = [h_B,l], \]
where $h_B$ is the derivation on $\csym\Sigma\gl{N}{A}^*$ given by contracting with the matrix
\[ \Sigma 1_A \otimes B\in\Sigma\gl{N}{A}. \]
This follows from the identities of Definition \ref{def_ainfunital}.

We have the following basic fact, which is a consequence of Weyl's Theorem on the representation theory of semisimple Lie algebras, cf. \cite{weibel} Proposition 9.10.3.

\begin{theorem} \label{thm_invariantqiso}
Given a unital $\ai$-algebra $A$, the inclusion of the invariants
\[ \left[\csym\Sigma\gl{N}{A}^*\right]^{\gl{N}{\gf}} \longrightarrow \csym\Sigma\gl{N}{A}^* \]
induces an isomorphism in Chevalley-Eilenberg cohomology.
\end{theorem}

\subsection{Morita Invariance}

Morita invariance is an important property of the Hochschild cohomology of associative algebras. Here we provide a formulation of this result that applies in the context of $\ai$-algebras.

Given an $\ai$-algebra $A$, denote the $\ai$-structures on $A$ and $\mat{N}{A}$ by $m^A$ and $m^{\mat{N}{A}}$ respectively. We define maps of complexes
\[ \Mor:\left(\ctenp \Sigma A^*, m^A\right) \rightleftarrows \left(\ctenp \Sigma \mat{N}{A}^*, m^{\mat{N}{A}}\right):\Res \]
as follows. The map $\Res$ is the restriction map, dual to the map that arises from embedding $A$ into the top left corner of $\mat{N}{A}$.

To define $\Mor$, consider the multilinear maps;
\begin{equation} \label{eqn_traceproduct}
t_k:\mat{N}{\gf}^{\otimes k} \to \gf, \qquad X_1,\ldots, X_k \mapsto \tr (X_1\cdots X_k);
\end{equation}
given by taking the trace of a product of matrices. The map $\Mor$ takes a $k$-multilinear form on $\Sigma A$ to a $k$-multilinear form on $\Sigma\mat{N}{A}=\Sigma A\otimes\mat{N}{\gf}$ by tensoring with the form $t_k$.

Since the multilinear forms \eqref{eqn_traceproduct} are cyclically symmetric, the maps $\Mor$ and $\Res$ descend to the cyclic complexes. We are now ready to provide the statement on the Morita invariance of cyclic cohomology.

\begin{theorem} \label{thm_moritainvariance}
Given a unital $\ai$-algebra $A$, the maps
\[ \Mor:\CycHoch{A} \rightleftarrows \CycHoch{\mat{N}{A}}:\Res \]
are quasi-isomorphisms, mutually inverse on homology, satisfying
\[ \Res\circ\Mor=\id. \]
\end{theorem}

\begin{proof}
Consider the descending filtrations of the above cyclic complexes given by filtering by the order $k$ of the tensors. These are complete bounded above filtrations compatible with the above maps. The differential on page zero of the spectral sequences is just the one induced from the differential $m_1$ on the $\ai$-algebra $A$. The cohomology $HA$ of the $\ai$-algebra $A$ with respect to $m_1$ is a unital strictly associative algebra. On page one of our spectral sequence the differential is just the usual cyclic differential coming from the strictly associative algebra structure on $HA$. Hence on page one of our spectral sequence the above maps $\Mor$ and $\Res$ induce an isomorphism by the classical theorem on Morita invariance, cf. \cite{loday} Theorem 2.4.6. The result now follows from the Eilenberg-Moore Comparison Theorem, cf. \cite{weibel} Theorem 5.5.11.
\end{proof}

\subsection{The Loday-Quillen-Tsygan Theorem} \label{sec_LQTthm}

The Loday-Quillen-Tsygan Theorem, discovered independently by Loday-Quillen \cite{lodayquillen} and Tsygan \cite{tsygan}, provides a description of the $N$-stable cohomology of $\gl{N}{A}$ in terms of the cyclic cohomology of $A$. It will play a central role in our exploration of large $N$ phenomena by allowing us to simultaneously handle all the possible phenomena arising for different values of $N$ through the use of an appropriate universal object.

Given an $\ai$-algebra, consider the composite map
\[ \csymp\left(\CycHoch{A}\right) \longrightarrow \csymp\left(\CycHoch{\mat{n}{A}}\right) \longrightarrow \csymp\left(\ChEilp{\gl{N}{A}}\right) \longrightarrow \ChEilp{\gl{N}{A}}, \]
where the left-hand map is induced by the map $\Mor$ from Theorem \ref{thm_moritainvariance}, the next is induced by \eqref{eqn_cyctoCE} and the last map applies multiplication (so an $n$-fold symmetric product goes to the product of the $n$ terms).
Since the multilinear forms \eqref{eqn_traceproduct} are $\gl{N}{\gf}$-invariant, this map lands in the $\gl{N}{\gf}$-invariants. This leads to a system of commuting maps,
\begin{equation} \label{eqn_LQTinvlim}
\xymatrix{ & \csymp\left(\CycHoch{A}\right) \ar[rd] \ar[ld] &  \\ \left[\ChEilp{\gl{N}{A}}\right]^{\gl{N}{\gf}} & & \left[\ChEilp{\gl{N+1}{A}}\right]^{\gl{N+1}{\gf}} \ar[ll] }
\end{equation}
in which the horizontal arrows are the restriction maps dual to the canonical inclusions of $\gl{N}{A}$ into $\gl{N+1}{A}$.

We are now ready to provide a statement of the Loday-Quillen-Tsygan Theorem, which we formulate for an $\ai$-algebra $A$. Although the original results \cite{lodayquillen, tsygan} of Loday-Quillen and Tsygan applied only to associative algebras, the arguments in these original sources prove the theorem for $\ai$-algebras, with no additions needed.

\begin{theorem} \label{thm_LQT}
The diagram \eqref{eqn_LQTinvlim} establishes $\csymp\left(\CycHoch{A}\right)$ as an inverse limit,
\[ \csymp\left(\CycHoch{A}\right) = \invlim{N}\left(\left[\ChEilp{\gl{N}{A}}\right]^{\gl{N}{\gf}}\right). \]
Moreover, the same applies to the cohomology of these complexes:
\begin{equation} \label{eqn_LQTinvlimhomology}
H^\bullet\left(\csymp\left(\CycHoch{A}\right)\right) = \invlim{N} H^\bullet\left(\left[\ChEilp{\gl{N}{A}}\right]^{\gl{N}{\gf}}\right).
\end{equation}
\end{theorem}

\begin{rem}
We emphasize that in the above statement, $A$ does \emph{not} need to be a \emph{unital} $\ai$-algebra.
\end{rem}

\section{The Batalin-Vilkovisky formalism and noncommutative geometry} \label{sec_BVNonComGeom}

In this section, we recall the basic framework of the Batalin-Vilkovisky formalism in both its commutative \cite{bvgeom} and noncommutative contexts \cite{baran, hamcompact, kontsympgeom}. Noncommutative symplectic geometry was introduced in \cite{kontsympgeom} by Kontsevich where it was used to describe the cohomology of the moduli space of Riemann surfaces. In the subsequent sections we will see how this noncommutative geometry appears in the large $N$ limit when we try to handle all ranks $N$ of the theory simultaneously.

\subsection{The commutative Batalin-Vilkovisky formalism} \label{sec_commBV}

We begin by briefly recalling the standard geometric framework of the Batalin-Vilkovisky formalism. Let $V$ be a graded vector space with a symplectic form $\innprod$ of odd degree. The inverse form $\innprod^{-1}$ on $V^*$ is defined by the commutative diagram
\begin{equation} \label{eqn_invinnprod}
\xymatrix{ & \gf \\ V\otimes V \ar[ur]^{\innprod} \ar[rr]^{D_l \otimes D_r} && V^*\otimes V^* \ar[ul]_{\innprod^{-1}}}
\end{equation}
where $D_l(y):=\langle y,- \rangle$ and $D_r(y):=\langle -,y \rangle$.
While the form $\innprod$ is skew-symmetric, the Koszul sign rule implies that the inverse form $\innprod^{-1}$ is symmetric.

From the inverse form we define a Poisson bracket $\{-,-\}$ of odd degree on both $\sym V^*$ and $\csym V^*$ by extending $\innprod^{-1}$ to the commutative algebras using the Leibniz rule.

\begin{defi} \label{def_poisson}
For a vector space $V$ with an odd symplectic form, we define
\[ \Poispoly{V}:= \sym V^*\]
and
\[ \Pois{V}:= \csym V^*\]
to be the Poisson algebras with odd Poisson bracket~$\{-,-\}$.
\end{defi}

This structure is sometimes known as a \emph{shifted Poisson algebra}.
Note that one can generate many other shifted Poisson algebras from these, as follows.
Pick a degree zero element $S$ such that $\{S,S\} = 0$, so that the degree 1 derivation $\{S, -\}$ determines a differential on the original algebra and hence a dg shifted Poisson algebra.
We will discuss the moduli of such shifted Poisson algebras in Sections \ref{sec_obsthy} and~\ref{sec_largeN}.

Every function $f$ in a Poisson algebra determines a Hamiltonian vector field $\{f, -\}$.
In our setting, with $V$ symplectic, this relationship provides a canonical isomorphism from $\symp V^*$ to the Lie algebra of symplectic vector fields (i.e., derivations on $\sym V^*$ preserving the symplectic form), as constant terms do not affect a Hamiltonian vector field.
Similarly, the symplectic vector fields acting on $\csym V^*$ are given by $\csymp V^*$.

There is a natural quantization of these shifted Poisson algebras $\Poispoly{V}$ and $\Pois{V}$ in the Batalin-Vilkovisky sense:
there is a deformation of the differential determined by the Poisson bracket.
To whit, the BV-Laplacian $\Delta$ is the unique operator on $\Poispoly{V}$ or $\Pois{V}$ satisfying
\begin{equation} \label{eqn_BVidentity}
\Delta(ab) = (\Delta a)b + (-1)^a a(\Delta b) +\{a,b\}
\end{equation}
and vanishing on $\sym^{< 2} V^*$ (i.e., on linear terms in $V^*$ and constant terms in~$\gf$).

\begin{defi}
For a vector space $V$ with an odd symplectic form, we define
\[ \ComBVpoly{V}:=\gf[\hbar]\otimes\sym V^* \]
and
\[ \ComBV{V}:=\gf[[\hbar]]\cotimes\csym V^* \]
to be the differential graded Lie algebra with odd bracket given by extending $\{-,-\}$ linearly with respect to $\hbar$ and with differential $\hbar\Delta$ defined in the same manner. Note that here the formal variable $\hbar$ has degree zero.
\end{defi}

This differential graded Lie algebra encodes the standard framework for quantization and the calculation of expectation values in the Batalin-Vilkovisky formalism.
It is not a dg commutative algebra because the differential is not a derivation for the commutative product.
Note, however, that the Lie bracket has odd degree and acts by derivations on the commutative product, which has even degree; this type of structure is sometimes referred to as a \emph{Beilinson-Drinfeld algebra}.
The quotients $\ComBVpoly{V}/\langle\hbar\rangle$ and $\ComBV{V}/\langle\hbar\rangle$ are shifted Poisson algebras because they have both a commutative product and a differential that is a derivation (because it is trivial).
Viewing the quotient as a way of setting $\hbar=0$,
we can see these algebras as ``quantizations'' of shifted Poisson algebras.

\subsection{The noncommutative Batalin-Vilkovisky formalism}

Now we describe a noncommutative counterpart to the geometric framework described above. The following definition is due to Kontsevich \cite{kontsympgeom}, which he introduced as a cyclic analogue of the Lie algebra of Hamiltonian vector fields on a symplectic vector space.

\begin{defi} \label{def_noncomBV}
Given a vector space $V$ with a symplectic form $\innprod$ of odd degree, we define a Lie bracket of odd degree on
\[ \NCHampoly{V}:=\bigoplus_{k=0}^\infty \big[(V^*)^{\otimes k}\big]_{\cycg{k}}, \]
and
\[ \NCHam{V}:=\prod_{k=0}^\infty \big[(V^*)^{\cotimes k}\big]_{\cycg{k}}, \]
by the formula
\[ \{(a_1\cdots a_m),(b_1\cdots b_n)\} := \sum_{i=1}^m\sum_{j=1}^n \pm\langle a_i,b_j \rangle^{-1} (a_{i+1}\cdots a_m a_1 \cdots a_{i-1} b_{j+1}\cdots b_n b_1\cdots b_{j-1}), \]
where $a_1,\ldots a_m, b_1,\ldots b_n\in V^*$ and the sign is determined canonically by the Koszul sign rule.
\end{defi}

There is also a Lie cobracket of odd degree on $\NCHampoly{V}$,
\[ \nabla:\NCHampoly{V}\to\left(\NCHampoly{V}\otimes\NCHampoly{V}\right)_{\symg{2}}, \]
given by the formula
\begin{equation} \label{eqn_noncomcobracket}
\nabla(a_1\cdots a_n):=\sum_{1\leq i < j\leq n}\pm\langle a_i,a_j \rangle^{-1} (a_{i+1}\cdots a_{j-1})\otimes(a_{j+1}\cdots a_n a_1 \cdots a_{i-1}).
\end{equation}

Again, the sign is determined by the Koszul sign rule.
The same formula equips $\NCHam{V}$ with a Lie cobracket of odd degree (technically this cobracket lands in the \emph{completed} symmetric tensor product, although we will ignore this distinction).
This construction of $\nabla$ was first described by Movshev in \cite{movshevcobracket}.
These structures turn $\NCHampoly{V}$ and $\NCHam{V}$ into Lie bialgebras.

Note that there is a subspace
\[ \NCHamPpoly{V}:=\bigoplus_{k=1}^\infty \big[(V^*)^{\otimes k}\big]_{\cycg{k}} \]
of $\NCHampoly{V}$, and there is a subspace
\[ \NCHamP{V}:=\prod_{k=1}^\infty \big[(V^*)^{\cotimes k}\big]_{\cycg{k}} \]
of $\NCHam{V}$.
This notation is parallel to the notation $\symp{V}$ for the augmentation ideal of $\sym{V}$.
We will use $\nu$ to denote the generator of the ``constant'' term $(V^*)^{\otimes 0}$ inside $\NCHampoly{V}$ or $\NCHam{V}$.
Thus,
\[ \NCHampoly{V} = \gf \nu \oplus \NCHamPpoly{V}, \]
and similarly for $\NCHam{V}$.

\subsection{Relation to cyclic infinity-structures} \label{sec_reltocyclicstruct}

We now explain how the constructions of the preceding subsection can be used to encode cyclic infinity-structures and their cohomology. As we will describe below, the Lie algebras defined in Definition \ref{def_poisson} and Definition \ref{def_noncomBV} may be identified with the cyclic and Chevalley-Eilenberg complexes when they are equipped with a differential that arises through the adjoint action of an element representing the cyclic infinity structure, see Equation \eqref{eqn_Hochbracketdiff} and \eqref{eqn_CEbracketdiff} below.

Let $A$ be a cyclic $\ai$-algebra with inner product $\innprod$ of odd degree. Since this bilinear form is symmetric it gives rise to a symplectic form, also denoted by $\innprod$, on $\Sigma A$.
\begin{equation} \label{eqn_sympform}
\xymatrix{ & \gf \\ A\otimes A \ar[ur]^{\innprod} \ar[rr]^{\Sigma \otimes \Sigma} && \Sigma A \otimes \Sigma A \ar[ul]_{\innprod}}
\end{equation}

Consider the family of cyclically antisymmetric multilinear forms \eqref{eqn_cyclicainfstructuremaps} on $A$ determined by the $\ai$-structure $m$ on $A$. These give rise to cyclically symmetric tensors on $\Sigma A$ and hence to an element of,
\[ \prod_{k=2}^\infty \big[(\Sigma A^*)^{\otimes k}\big]^{\cycg{k}} \cong \prod_{k=2}^\infty \big[(\Sigma A^*)^{\otimes k}\big]_{\cycg{k}}\subset\NCHam{\Sigma A}. \]
We denote the corresponding element of $\NCHam{\Sigma A}$ by $\tilde{m}$. Note, importantly, that the above isomorphism is defined so that it takes a cyclically invariant $k$-fold tensor $\alpha_k$ on the left to the cyclically coinvariant tensor $\frac{1}{k}\alpha_k$ on the right. The homotopy coherence relations are then equivalent to the equation $\{\tilde{m},\tilde{m}\}=0$. In this way, the moduli space of cyclic $\ai$-structures on $A$ may be identified with the Maurer-Cartan moduli space that is associated to the Lie subalgebra of $\NCHam{\Sigma A}$ consisting of quadratic and higher order elements.

We can also describe the cohomology of cyclic infinity-structures using the constructions of the preceding subsection. We note that $\NCHamP{\Sigma A}$ is the underlying space of the cyclic complex and that the Hochschild differential satisfies
\[ m(\alpha) = -\{\tilde{m},\alpha\}, \quad \alpha\in\prod_{k=1}^\infty \big[(\Sigma A^*)^{\otimes k}\big]_{\cycg{k}}. \]
Hence the cyclic complex may be identified within the BV formalism as
\begin{equation} \label{eqn_Hochbracketdiff}
\CycHoch{A} = \left(\NCHamP{\Sigma A}, -\{\tilde{m},-\}\right),
\end{equation}
a relationship that plays a key role in this paper.

A story similar to that described above holds for cyclic $\li$-structures. Given a cyclic $\li$-algebra $\mathfrak{g}$ the family \eqref{eqn_cycliclinfstructuremaps} of antisymmetric tensors on $\mathfrak{g}$ that arises from the $\li$-structure $l$ on $\mathfrak{g}$ yields an element of,
\[ \prod_{k=2}^\infty \big[(\Sigma \mathfrak{g}^*)^{\otimes k}\big]^{\symg{k}} \cong \prod_{k=2}^\infty \big[(\Sigma \mathfrak{g}^*)^{\otimes k}\big]_{\symg{k}}\subset\csym\Sigma\mathfrak{g}^*. \]
We denote the corresponding element in $\csym\Sigma\mathfrak{g}^*$ by $\tilde{l}$. Again, we note that here the above isomorphism is defined so that it takes an invariant $k$-fold tensor $\alpha_k$ on the left to the coinvariant tensor $\frac{1}{k!}\alpha_k$ on the right. The homotopy Jacobi identities are then equivalent to the equation $\{\tilde{l},\tilde{l}\}=0$. In this way, the moduli space of cyclic $\li$-structures on $\mathfrak{g}$ may be identified with the Maurer-Cartan moduli space of quadratic and higher order elements in the Lie algebra $\csym\Sigma\mathfrak{g}^*$. Via the BV formalism, the Chevalley-Eilenberg complex may then be identified similarly as
\begin{equation} \label{eqn_CEbracketdiff}
\ChEilp{\mathfrak{g}}:=\left(\csymp\Sigma\mathfrak{g}^*,-\{\tilde{l},-\}\right),
\end{equation}
another relationship that plays a key role in this paper.

\subsection{The action of $\gl{N}{\gf}$}

We now explain how the action of $\gl{N}{\gf}$ described in Section \ref{sec_LinfCE} is compatible with the structures of the Batalin-Vilkovisky formalism. Recall that given a cyclic $\ai$-algebra $A$ whose inner product has odd degree, the Lie algebra $\gl{N}{\gf}$ acts on $\csym\Sigma\gl{N}{A}^*$. The latter has the structure of a BV-algebra, cf. Equation \eqref{eqn_BVidentity}, where the pairing on $\gl{N}{A}$ combines the pairing on $A$ with the trace pairing on matrices. The action by $\gl{N}{\gf}$ respects this structure in the following sense.

\begin{lemma} \label{lem_BVinvariant}
For all $f,g\in\csym\Sigma\gl{N}{A}^*$ and $B\in\gl{N}{\gf}$:
\begin{enumerate}
\item \label{item_BVinvariantidentities1}
$B\cdot(fg) = (B\cdot f)g + f(B\cdot g)$,
\item \label{item_BVinvariantidentities2}
$B\cdot\{f,g\} = \{Bf,g\} + \{f,Bg\}$,
\item \label{item_BVinvariantidentities3}
$\Delta(B\cdot f) = B\cdot\Delta f$.
\end{enumerate}
\end{lemma}

\begin{proof}
\eqref{item_BVinvariantidentities1} is tautological and since $\{-,-\}$ is a Poisson bracket and hence acts by derivations on $\csym\Sigma\gl{N}{A}^*$, it suffices to verify \eqref{item_BVinvariantidentities2} for $f$ and $g$ in $\Sigma\gl{N}{A}^*$; this follows as the trace vanishes on commutators. To prove \eqref{item_BVinvariantidentities3}, we note that it follows from \eqref{item_BVinvariantidentities1}, \eqref{item_BVinvariantidentities2} and Equation \eqref{eqn_BVidentity} that $[\Delta,B]$ is a derivation on $\csym\Sigma\gl{N}{A}^*$ and hence must be zero since it vanishes on $\Sigma\gl{N}{A}^*$.
\end{proof}

It follows that the $\gl{N}{\gf}$-invariants form a BV-subalgebra of $\csym\Sigma\gl{N}{A}^*$ and hence that,
\[ \ComBV{\Sigma\gl{N}{A}}^{\gl{N}{\gf}} = \gf[[\hbar]]\cotimes\left[\csym\Sigma\gl{N}{A}^*\right]^{\gl{N}{\gf}} \]
is a differential graded Lie subalgebra of $\ComBV{\Sigma\gl{N}{A}}$.

\section{Morita Invariance and the Loday-Quillen-Tsygan Theorem in the Batalin-Vilkovisky formalism} \label{sec_MinvLQTBV}

In this section we begin to describe how Morita invariance and the Loday-Quillen-Tsygan Theorem fit into the homological algebra of the Batalin-Vilkovisky formalism.

\subsection{Relating the commutative and noncommutative formalisms} \label{sec_relcomnoncom}

We are now in a position to describe a key relationship for our paper.
We will show that, first, there is a canonical map from the noncommutative object built from an odd symplectic vector space $V$ to the commutative BV object built from $V$ and, second, that a noncommutative BV quantization determines canonically a commutative BV quantization.
These relationships provide the first step toward a quantized Loday-Quillen-Tsygan theorem.

Both $\NCHampoly{V}$ and $\Poispoly{V}$ are, as graded vector spaces, quotients of the tensor algebra $\ten{V^*}$.
Hence there is a canonical quotient map
\begin{equation}\label{eqn_HamPoispoly}
\sigma: \NCHampoly{V} = \bigoplus_{k=0}^\infty \big[(V^*)^{\otimes k}\big]_{\cycg{k}} \longrightarrow \bigoplus_{k=0}^\infty \big[(V^*)^{\otimes k}\big]_{\symg{k}} = \Poispoly{V}
\end{equation}
sending a cyclic word in $k$ letters to a symmetric word in $k$ letters, because $\cycg{k} \subset \symg{k}$.
Note that $\nu$ goes to 1.
By direct inspection, $\sigma$ is a map of Lie algebras with odd bracket:
on both sides, the bracket arises by pairing off two elements in a word in all possible ways, whether the word is cyclic or symmetric.

Similarly, there is a canonical map of Lie algebras
\[ \widehat{\sigma}: \NCHam{V} = \prod_{k=0}^\infty \big[(V^*)^{\cotimes k}\big]_{\cycg{k}} \longrightarrow \prod_{k=0}^\infty \big[(V^*)^{\cotimes k}\big]_{\symg{k}} = \Pois{V}, \]
by replacing direct sum with product.

Any map $\mathfrak{g} \to P$ from a Lie algebra to a Poisson algebra determines canonically a map of Poisson algebras $\sym \mathfrak{g} \to P$, where the commutative algebra $\sym \mathfrak{g} $ generated by $\mathfrak{g} $ has a canonical Poisson bracket given by extending the Lie bracket of $\mathfrak{g}$ using the Leibniz rule and the Poisson map is the canonical map of commutative algebras.
Hence we have maps of shifted Poisson algebras
\begin{equation}\label{eqn_PoisHamPoispoly}
\sigma_P: \sym\left({\NCHampoly{V}}\right) \longrightarrow \Poispoly{V}
\end{equation}
and
\begin{equation}\label{eqn_PoisHamPois}
\widehat{\sigma}_P: \bigoplus_{i=0}^\infty\left[\NCHam{V}^{\cotimes i}\right]_{\symg{i}} \longrightarrow \Pois{V}.
\end{equation}
Observe that
\[ \sym(\NCHampoly{V}) = \gf[\nu] \otimes \sym(\NCHamPpoly{V}), \]
and so we can describe these maps concretely as
\[ \nu^m (a_{11}\cdots a_{1k_1})\cdots(a_{n1}\cdots a_{nk_n}) \mapsto a_{11}\cdots a_{1k_1}\cdots a_{n1}\cdots a_{nk_n}, \]
sending a symmetric product of cyclic words to the symmetric product of all the terms.
These maps send $\nu$ to 1.
Note that there is a commuting square for these maps induced by the inclusions $\NCHampoly{V} \hookrightarrow \NCHam{V}$ and $\Poispoly{V} \hookrightarrow \Pois{V}$.

\begin{rem}
We have worked only with the direct sum for the domain of \eqref{eqn_PoisHamPois}.
One cannot extend the map to the completion $\csym(\NCHam{V})$ as it is not continuous; that is the map sends $\nu$ to 1 and hence, for instance, any formal power series $f(\nu)$ from $\csym(\NCHam{V})$ would be sent to $f(1)$, which is (in general) not well-defined.
\end{rem}

These maps relate two examples of the commutative BV formalism at the classical level.
It is natural to ask whether we can quantize these maps, and we will formulate two answers.
The first quantization is quite straightforward.

\begin{defi}
Given a vector space $V$ with a symplectic form $\innprod$ of odd degree, we define
\[ \NCBVpoly{V} = \gf[\hbar] \otimes \sym\left(\NCHampoly{V}\right) \]
to be the differential graded Lie algebra with odd bracket given by extending the bracket on $\sym(\NCHampoly{V})$ linearly over $\hbar$ and with differential $\hbar(\delta + \nabla)$,
where $\delta$ denotes the Chevalley-Eilenberg differential associated to the Lie algebra $\NCHampoly{V}$ (it arises by extending the bracket as a coderivation on the commutative coalgebra $\sym\left(\NCHampoly{V}\right)$) and the cobracket $\nabla$ defined by \eqref{eqn_noncomcobracket} on $\NCHampoly{V}$ is extended to $\sym\left(\NCHampoly{V}\right)$ as a derivation.

Similarly, let
\[ \NCBV{V} = \gf[\hbar] \otimes \left(\bigoplus_{i=0}^\infty\left[\NCHam{V}^{\cotimes i}\right]_{\symg{i}}\right) \]
denote the corresponding construction.
\end{defi}

It follows from the Lie bialgebra axioms for $\NCHampoly{V}$ that the differential has square zero and acts as a derivation for the bracket.
Note that it is well-known that $\sym\left(\NCHampoly{V}\right)$ is canonically a BV-algebra when equipped with the Chevalley-Eilenberg differential $\delta$. This remains true when $\delta$ is replaced by $(\delta+\nabla)$, as above, as $\nabla$ is a derivation.

\begin{prop}\label{first quantum sigma}
The map \eqref{eqn_PoisHamPoispoly} extends $\hbar$-linearly to a map
\[ \sigma_\hbar: \NCBVpoly{V} \longrightarrow \ComBVpoly{V} \]
of differential graded Lie algebras.

Likewise, the map \eqref{eqn_PoisHamPois} extends $\hbar$-linearly to a map
\[ \widehat{\sigma}_\hbar: \NCBV{V} \longrightarrow \ComBV{V} \]
of differential graded Lie algebras.
\end{prop}

The statement of this result looks rather technical, but it has an interesting interpretation.
It says that there is a non-obvious ``commutative'' BV quantization of the Poisson algebra $\sym(\NCHampoly{V})$ that maps by passing to the usual ``commutative'' BV quantization $\ComBVpoly{V}$.
In conjunction with the Loday-Quillen-Tsygan Theorem,
this relationship will connect the noncommutative BV formalism with the commutative formalism at all ranks~$N$.

\begin{proof}
We have already seen that, modulo $\hbar$, these are maps of Lie algebras,
and so, as the maps are $\hbar$-linear extensions, these maps are still Lie algebra maps.
What we need to check is that they are cochain maps.

We begin by noting that for all $x=(x_1\cdots x_k)\in\NCHampoly{V}$,
\begin{equation} \label{eqn_cobracketisdelta}
\Delta\sigma_P(x) = \sigma_P\nabla(x).
\end{equation}
This is a consequence of Equation \eqref{eqn_BVidentity}, since it follows from this that $\Delta$ contracts all pairs of tensors $x_i\in V^*$ in $x$ using the inverse bilinear form $\innprod^{-1}$.

It suffices now to analyze the differentials on simple terms that span both the domain and range.
Consider a symmetric product of cyclic words
\[ \alpha = a_1\cdot a_2\cdots a_n, \quad a_i\in\NCHampoly{V} \]
in $\NCBVpoly{V}$.
Now from Equation \eqref{eqn_BVidentity} and the fact that \eqref{eqn_PoisHamPoispoly} is a map of Poisson algebras,
it follows that
\begin{displaymath}
\begin{split}
\Delta\sigma_{\hbar}(\alpha) &= \Delta(\sigma_P(a_1)\cdots \sigma_P(a_n)) \\
=& \sum_{1\leq i < j \leq n}\pm\{\sigma_P(a_i),\sigma_P(a_j)\}\cdot\sigma_P(a_1)\cdots \widehat{\sigma_P(a_i)}\cdots\widehat{\sigma_P(a_j)}\cdots \sigma_P(a_n) \\
&+ \sum_{1\leq i \leq n}\pm\Delta(\sigma_P(a_i))\cdot\sigma_P(a_1)\cdots\widehat{\sigma_P(a_i)}\cdots \sigma_P(a_n) \\
=&\sigma_P\delta(\alpha) + \sigma_P\nabla(\alpha) = \sigma_{\hbar}(\delta+\nabla)(\alpha)
\end{split}
\end{displaymath}
where on the last line we have used Equation \eqref{eqn_cobracketisdelta}.
\end{proof}

There is another quantization, motivated by the relationship between the noncommutative BV formalism and the moduli space of Riemann surfaces and with topological field theory.
An extensive discussion can be found in~\cite{hamcompact}, which motivates many of the choices and parameters below.

\begin{defi} \label{def_noncomquantumBV}
For a vector space $V$ having a symplectic form of odd degree, define
\begin{displaymath}
\begin{split}
\NonComBV{V} &:= \gf[[\gamma]]\cotimes\csymp(\NCHam{V}), \\
&= \Big[\nu\gf[[\gamma,\nu]]\Big] \times \Big[\gf[[\gamma,\nu]]\cotimes\csymp(\NCHamP{V})\Big].
\end{split}
\end{displaymath}
It is a differential graded Lie algebra with differential
\[ \Delta_K := \nabla + \gamma \,\delta \]
and with bracket $\{-,-\}$, where we extend the structures on $\csymp(\NCHam{V})$ linearly with respect to $\gamma$.
\end{defi}

\begin{prop}
The map $\sigma_K: \NonComBV{V} \to \ComBV{V}$ given by
\[ \sigma_K\left(\gamma^i\nu^j(a_{11}\cdots a_{1k_1})\cdots(a_{n1}\cdots a_{nk_n}) \right) = \hbar^{2i+j+n-1}a_{11}\cdots a_{1k_1}\cdots a_{n1}\cdots a_{nk_n} \]
is a map of differential graded Lie algebras.
\end{prop}

\begin{proof}
The map $\sigma_K$ can be understood as being built out of a couple of pieces.
First, it sends $\gamma$ to $\hbar^2$.
Second, it weights a symmetric product of $n$ cyclic words from $\NCHam{V}$ by a factor of $\hbar^{n-1}$, regardless of whether the words come from $\NCHamP{V}$ or are the `empty' word $\nu$.
In the proof of Proposition \ref{first quantum sigma}, we saw
\[ \Delta\sigma_P = \sigma_P(\nabla + \delta). \]
Note that $\nabla$ increases the symmetric degree by one, whereas $\delta$ decreases it by one. Hence it follows that for $\alpha\in\sym^n(\NCHampoly{V})$,
\begin{displaymath}
\begin{split}
\sigma_K(\nabla\alpha+\gamma\delta\alpha) &= \hbar^n\sigma_P\nabla\alpha + \hbar^2\hbar^{n-2}\sigma_P\delta\alpha \\
&= \hbar^n\Delta\sigma_P\alpha = \hbar\Delta\sigma_K\alpha.
\end{split}
\end{displaymath}
Note also that the Lie bracket decreases the total symmetric degree by one and so a similar calculation demonstrates that $\sigma_K$ also respects the Lie brackets.
\end{proof}

There are further variations on quantization but we do not need them in this paper.

\subsection{Morita invariance and quantizing the Loday-Quillen-Tsygan maps}

We now turn to putting the construction just developed into dialogue with the Loday-Quillen-Tsygan Theorem.
As a first step note that if $V$ is a vector space with odd symplectic pairing $\innprod$, then $V \otimes \mat{N}{\gf}$ has an odd symplectic pairing:
\[ \langle a \otimes X , b \otimes Y \rangle_N = \langle a, b \rangle\, \tr(XY). \]
Consider now the map
\[ \Mor_+:\NCHamP{V}\to\NCHamP{V\otimes\mat{N}{\gf}}. \]
arising from the map $\Mor$ of Theorem \ref{thm_moritainvariance},
which is defined by tensoring with the multilinear forms~\eqref{eqn_traceproduct}.
(Recall that this map $\Mor$ is built from traces of products of matrices.)
We can extend it to a map
\[ \Mor:\NCHam{V}\to\NCHam{V\otimes\mat{N}{\gf}}. \]
by sending $\nu$ to $N\nu$.
(In Sections~\ref{sec_random} and \ref{sec_largeN}, this scaling of $\nu$ will play a critical role.)
Let $\Mor$ denote as well the map from $\NCHampoly{V}$ to $\NCHampoly{V\otimes\mat{N}{\gf}}$.

This choice has a nice feature.

\begin{lemma} \label{lem_Moritamapbialgebras}
The map $\Mor$ is a map of Lie bialgebras.
\end{lemma}

This claim may be verified by direct computation -- the calculations are similar to those performed in Section \ref{Phi map on matrices}, see in particular Equation \eqref{eqn_calcOTFTnulltensor} -- although we will not take this approach here since we will provide an alternative proof based on a more general fact about tensoring $V$ with an associative Frobenius algebra $M$. This particular generalization is quite appealing: it can be seen as coupling the BV theory arising from $V$ to an open topological field theory associated to~$M$. The proof of Lemma \ref{lem_Moritamapbialgebras} and this generalization will be deferred to Section \ref{Phi map on matrices}.

As a consequence of the above lemma, we see that there is a canonical map
\begin{equation} \label{eqn_Moritapoisson}
\Mor_P: \sym(\NCHampoly{V}) \longrightarrow \sym(\NCHampoly{V\otimes\mat{N}{\gf}})
\end{equation}
of shifted Poisson algebras.
Similarly, replacing $\NCHampoly{V}$ with $\NCHam{V}$, we have a map of shifted Poisson algebras
\[ \widehat{\Mor}_P: \bigoplus_{i=0}^\infty\left[\NCHam{V}^{\cotimes i}\right]_{\symg{i}} \longrightarrow \bigoplus_{i=0}^\infty\left[\NCHam{V\otimes\mat{N}{\gf}}^{\cotimes i}\right]_{\symg{i}}. \]
Moreover, as the multilinear maps~\eqref{eqn_traceproduct} are $\gl{N}{\gf}$-invariant,
the image lands in the $\gl{N}{\gf}$-invariant subspace of the target.

Combining the above with \eqref{eqn_PoisHamPoispoly} and \eqref{eqn_PoisHamPois}, we obtain the following.

\begin{lemma} \label{lem_shiftpoismap}
The maps
\[ \sigma_P \circ \Mor_P: \sym(\NCHampoly{V}) \longrightarrow \Poispoly{V\otimes\mat{N}{\gf}}^{\gl{N}{\gf}} \subset \Poispoly{V\otimes\mat{N}{\gf}} \]
and
\[ \widehat{\sigma}_P \circ \widehat{\Mor}_P: \bigoplus_{i=0}^\infty\left[\NCHam{V}^{\cotimes i}\right]_{\symg{i}} \longrightarrow \Pois{V\otimes\mat{N}{\gf}}^{\gl{N}{\gf}} \subset \Pois{V\otimes\mat{N}{\gf}} \]
are maps of shifted Poisson algebras.
\end{lemma}

We interpret these maps as saying that the noncommutative theory arising from $V$ encodes the conjugation-invariant piece of the commutative theory arising from~$V \otimes \mat{N}{\gf}$ for every rank~$N$.
In other words, it captures something common to all such theories.

It is natural to ask whether this relationship lifts to the quantizations (of both kinds), which indeed it does. Denote by $\Mor_{\hbar}$ the $\hbar$-linear extension of $\Mor_P$ to $\NCBVpoly{V}$ and likewise for $\widehat{\Mor}_P$.

\begin{theorem} \label{thm_matrixdglamap}
The maps
\[ \sigma_\hbar \circ \Mor_{\hbar}: \NCBVpoly{V} \longrightarrow \ComBVpoly{V\otimes\mat{N}{\gf}}^{\gl{N}{\gf}} \subset \ComBVpoly{V\otimes\mat{N}{\gf}} \]
and
\[ \widehat{\sigma}_\hbar \circ \widehat{\Mor}_{\hbar}: \NCBV{V} \longrightarrow \ComBV{V\otimes\mat{N}{\gf}}^{\gl{N}{\gf}} \subset \ComBV{V\otimes\mat{N}{\gf}} \]
are maps of differential graded Lie algebras.
\end{theorem}

\begin{proof}
It remains only to check that these are cochain maps.
As we know $\sigma_\hbar$ is a cochain map, we only need to verify that $\Mor_{\hbar}$ intertwines the differentials, which follows directly from Lemma \ref{lem_Moritamapbialgebras}.
\end{proof}

Similarly, there is the following result for the second quantization, whose proof we also defer to Section \ref{Phi map on matrices}.
Consider the $\gamma$-linear extension of the map \eqref{eqn_Moritapoisson} to the complete space $\NonComBV{V}$ and denote this map by
\begin{equation} \label{eqn_OTFTmatrixdglamap}
\widehat{\Mor}_{\gamma,\nu}:\NonComBV{V} \longrightarrow \NonComBV{V\otimes \mat{N}{\gf}}.
\end{equation}

\begin{theorem} \label{thm_OTFTmatrixdglamap}
The map $\widehat{\Mor}_{\gamma,\nu}$ is a map of differential graded Lie algebras. Consequently the composition
\[ \sigma_K\circ\widehat{\Mor}_{\gamma,\nu}:\NonComBV{V}\longrightarrow\ComBV{V\otimes\mat{N}{\gf}}^{\gl{N}{\gf}} \]
is also a map of differential graded Lie algebras.
\end{theorem}

\subsection{The quantized Loday-Quillen-Tsygan theorems}

We now extend the constructions just developed to the much more interesting setting of cyclic $\ai$- and $\li$-algebras.

Recall from \eqref{eqn_Hochbracketdiff} that for a cyclic $\ai$-algebra $A$ whose inner product is odd, we have
\[ \CycHoch{A} = \left(\NCHamP{\Sigma A}, -\{\tilde{m},-\}\right) \]
where $\tilde{m}\in\NCHamP{\Sigma A}$ represents the cyclic $\ai$-structure.
If we consider the underlying symplectic vector space $V = \Sigma A$ and ignore the $\ai$-structures,
then $\NCHamP{\Sigma A} = \NCHamP{V}$ is precisely what we've been working with so far.
Similarly, by \eqref{eqn_CEbracketdiff}, we know that for a cyclic $\li$-algebra $\mathfrak{g}$ with odd inner product,
\[ \ChEilp{\mathfrak{g}} =\left(\csymp\Sigma\mathfrak{g}^*,-\{\tilde{l},-\}\right), \]
where $\tilde{l}\in\csymp\Sigma\mathfrak{g}^*$ represents the cyclic $\li$-structure.
If we consider the underlying symplectic vector space $V = \Sigma \mathfrak{g}$,
then $\csym\Sigma\mathfrak{g}^* = \Pois{V}$ is also precisely what we've been working with so far.
Recall from Definition \ref{def_commutatoralgebra} that every cyclic $\ai$-algebra determines a commutator cyclic $\li$-algebra.

\begin{rem}
Note here that the completed versions are used, namely $\NCHamP{\Sigma A}$ and $\csymp\Sigma\mathfrak{g}^*$, because an arbitrary $\ai$- or $\li$-algebra has infinitely many operations and so the product of cyclic or symmetric powers is needed, rather than the direct sum.
If the algebra has only finitely many nontrivial operations, however, one can work with $\NCHamPpoly{\Sigma A}$ and $\symp\Sigma\mathfrak{g}^*$.
Here we will state the maps in the completed setting, but the reader can formulate the uncompleted versions as needed.
\end{rem}

The Lie algebra morphism $\sigma$ defined by \eqref{eqn_HamPoispoly} induces a map between the corresponding Maurer-Cartan sets; that is to say, following the discussion in Section \ref{sec_reltocyclicstruct}, that it maps cyclic $\ai$-structures to cyclic $\li$-structures. This is just the commutator algebra construction of Definition \ref{def_commutatoralgebra}; that is to say that $\tilde{l}:=\sigma(\tilde{m})$ is the commutator cyclic $\li$-structure associated to $\tilde{m}$.

Similarly, given a cyclic $\ai$-algebra $A$, the cyclic $\ai$-structures $\tilde{m}_A$ and $\tilde{m}_{\mat{N}{A}}$ on $A$ and $\mat{N}{A}$ respectively (cf. Definition \ref{def_ainfmatrices}) correspond under the maps $\Mor$ and $\Res$ of Theorem~\ref{thm_moritainvariance}:
\[ \Mor(\tilde{m}_A) = \tilde{m}_{\mat{N}{A}}, \qquad \Res(\tilde{m}_{\mat{N}{A}}) = \tilde{m}_{A}. \]

In light of this, turning on a nontrivial $\ai$-structure $\tilde{m}_A$ on $A$ with $V = \Sigma A$, one finds that the maps constructed in Theorems~\ref{thm_matrixdglamap} and~\ref{thm_OTFTmatrixdglamap} are extensions of the Loday-Quillen-Tsygan map of Section~\ref{sec_LQTthm}. More precisely, $\NCHam{V}$ sits inside both of the Lie algebras $\NCBV{V}$ and $\NonComBV{V}$ as a Lie subalgebra. Hence the adjoint map $\{\tilde{m}_A,-\}$ determines a differential on them. Similarly, if $\tilde{l}_{\gl{N}{A}}$ denotes the commutator $\li$-structure on matrices, it follows from Lemma \ref{lem_BVinvariant}\eqref{item_BVinvariantidentities2} that $\{\tilde{l}_{\gl{N}{A}},-\}$ defines a differential on $\ComBV{\Sigma\gl{N}{A}}^{\gl{N}{\gf}}$ since $\tilde{l}_{\gl{N}{A}}$ is $\gl{N}{\gf}$-invariant.

We have the following refinements of Theorems \ref{thm_matrixdglamap} and \ref{thm_OTFTmatrixdglamap}, explaining how these theorems may be viewed as quantizations of the Loday-Quillen-Tsygan Theorem.

\begin{theorem}\label{thm hbar LQT}
Given a (not necessarily unital) cyclic $\ai$-algebra $A$ whose inner product is odd and an integer $N \geq 1$, the map
\[ \widehat{\sigma}_\hbar \circ \widehat{\Mor}_{\hbar}: \NCBV{\Sigma A} \to \ComBV{\Sigma\gl{N}{A}}^{\gl{N}{\gf}} \]
not only intertwines the differentials and Lie brackets defined in Sections \ref{sec_commBV} and \ref{sec_relcomnoncom}, but also the differentials $\{\tilde{m}_A,-\}$ and $\{\tilde{l}_{\gl{N}{A}},-\}$ defined above. Modulo $\hbar$ and $\nu$, this map becomes the Loday-Quillen-Tsygan map of Theorem~\ref{thm_LQT}.
\end{theorem}

\begin{rem}
We note that if $A$ has only finitely many nonzero operations (i.e., $\tilde{m}_n = 0$ for all $n$ sufficiently large), then the above may be replaced by the map
\begin{equation}\label{eqn quant lqt poly}
\sigma_\hbar \circ \Mor_{\hbar}: \NCBVpoly{\Sigma A} \to \ComBVpoly{\Sigma\gl{N}{A}}^{\gl{N}{\gf}}.
\end{equation}
\end{rem}

\begin{proof}
A simple proof may be given as follows. For all $x\in\NCBV{\Sigma A}$,
\begin{displaymath}
\begin{split}
\widehat{\sigma}_\hbar\widehat{\Mor}_{\hbar}\{\tilde{m}_A,x\} &= \{\widehat{\sigma}_\hbar\widehat{\Mor}_{\hbar}(\tilde{m}_A),\widehat{\sigma}_\hbar\widehat{\Mor}_{\hbar}(x)\} \\
&= \{\widehat{\sigma}_\hbar(\tilde{m}_{\mat{N}{A}}),\widehat{\sigma}_\hbar\widehat{\Mor}_{\hbar}(x)\} = \{\tilde{l}_{\gl{N}{A}},\widehat{\sigma}_\hbar\widehat{\Mor}_{\hbar}(x)\}.
\end{split}
\end{displaymath}
\end{proof}

Likewise, we have the following.

\begin{theorem}
Given a (not necessarily unital) cyclic $\ai$-algebra $A$ whose inner product is odd and an integer $N \geq 1$, the map
\begin{equation} \label{eqn_LQTdglamaps}
\sigma_K\circ\widehat{\Mor}_{\gamma,\nu}: \NonComBV{\Sigma A}\longrightarrow\ComBV{\Sigma\gl{N}{A}}^{\gl{N}{\gf}}
\end{equation}
not only intertwines the differentials and Lie brackets defined in Sections \ref{sec_commBV} and \ref{sec_relcomnoncom}, but also the differentials $\{\tilde{m}_A,-\}$ and $\{\tilde{l}_{\gl{N}{A}},-\}$ defined above.
\end{theorem}

\begin{rem} \label{rem_BVdiffdeform}
It should be noted that for $\ComBV{\Sigma\gl{N}{A}}$, both $\Delta$ and $\{\tilde{l}_{\gl{N}{A}},-\}$ are differentials on this Lie algebra, but $(\Delta+\{\tilde{l}_{\gl{N}{A}},-\})$ is typically \emph{not}. This will hold only if
\[ \Delta\tilde{l}_{\gl{N}{A}} + \left\{\tilde{l}_{\gl{N}{A}},\tilde{l}_{\gl{N}{A}}\right\} = 0, \]
which in the case above is equivalent to
\begin{equation} \label{eqn_qmecondition}
\Delta\tilde{l}_{\gl{N}{A}}=0
\end{equation}
since the right-hand term already vanishes. Similar remarks apply to both $\NCBV{\Sigma A}$ and $\NonComBV{\Sigma A}$.

There is one particular case where \eqref{eqn_qmecondition} applies, namely when the cyclic $\ai$-structure consists of just a differential $d$ with no multiplication or higher order operations, so that $\tilde{m}_A$ is a purely quadratic term of even degree. In this case we can deform the BV-Laplacian on $\ComBV{\Sigma\gl{N}{A}}$ to $(d+\hbar\Delta)$. Again, similar remarks apply to both $\NCBV{\Sigma A}$ and $\NonComBV{\Sigma A}$.
\end{rem}

\subsection{Map of differential graded Lie algebras}

In this section we recall from \cite{hamNCBVform} the construction of a map of differential graded Lie algebras that appears in the noncommutative BV formalism and arises from a two-dimensional open topological field theory. The latter entities are known to correspond to Frobenius algebras. We will use this construction to give a proof of Lemma \ref{lem_Moritamapbialgebras} and Theorem \ref{thm_OTFTmatrixdglamap} by applying this construction to the Frobenius algebra $\mat{N}{\gf}$ of $N\times N$ matrices. It is in this case that we will see a remarkable simplification in the structure maps of our open topological field theory which will ultimately allow us to connect them to the maps \eqref{eqn_LQTinvlim} that appear in the Loday-Quillen-Tsygan Theorem.

Let $M$ be a Frobenius algebra which, for our purposes, we assume to be concentrated in degree zero. Let $\innprod$ denote the symmetric inner product on $M$ and
\[ \innprod^{-1} = x_i\otimes y^i \in M\otimes M \]
be the inverse inner product, where the repeated index $i$ indicates a summation. We will define a family of tensors,
\[ \mu^{g,b}_{k_1,\ldots,k_m}:M^{\otimes k_1}\otimes\cdots\otimes M^{\otimes k_m} \to \gf; \qquad g,b\geq 0, k_1,\ldots k_m >0. \]
The tensor $\mu^{g,b}_{k_1,\ldots,k_m}$ is the structure map associated, by the open topological field theory determined by $M$, to a genus $g$ surface with $b$ free boundary components and $m$ remaining boundary components, each of which respectively contain $k_i$ parameterized embedded intervals. We will provide below a concrete description of these tensors, taken from \cite{hamNCBVform}. For this reason, we do not labor to explain precisely what is meant by the term `open topological field theory' in this context, since a precise formulation in terms of modular operads is provided in \cite{chulazOTFT} as well as \cite{hamNCBVform} based on the Atiyah-Segal axioms \cite{atiyahTQFT}, and since it will be sufficient for our purposes here to merely recite the results of \cite{hamNCBVform}. Nonetheless, we feel obliged to explain the conceptual origins of these tensors.

We begin by defining
\begin{multline} \label{eqn_OTFTnulltensor}
\mu^{0,0}_{k_1,\ldots,k_m}\left(c_{11},\ldots c_{1k_1};\ldots;c_{m1},\ldots,c_{mk_m}\right) := \\
t_m\left(x_{i_m},\ldots,x_{i_1}\right)t_{k_1+\cdots+k_m+m}\left(y^{i_1},c_{11},\ldots c_{1k_1},\ldots,y^{i_m},c_{m1},\ldots,c_{mk_m}\right),
\end{multline}
where the repeated indices indicate a summation and the multilinear form $t_k$ is defined by
\[ t_k:M^{\otimes k}\to\gf, \qquad t_k(c_1,\ldots c_k) = \langle c_1\cdots c_{k-1},c_k \rangle. \]
Next we define maps $\beta,\gamma:M\to M$ by
\begin{equation} \label{eqn_OTFTfreebdry}
\beta(c):= x_i y^i c
\end{equation}
and
\begin{equation} \label{eqn_OTFTfreegenus}
\gamma(c):= x_i x_j y^i y^j c.
\end{equation}
Finally we may define
\begin{equation} \label{eqn_OTFTtensor}
\mu^{g,b}_{k_1,\ldots,k_m}\left(c_{11},\ldots c_{1k_1};\ldots;c_{m1},\ldots,c_{mk_m}\right) := \mu^{0,0}_{k_1,\ldots,k_m}\left(\beta^b\gamma^g(c_{11}),\ldots c_{1k_1};\ldots;c_{m1},\ldots,c_{mk_m}\right).
\end{equation}
In fact formula \eqref{eqn_OTFTtensor} remains unchanged regardless of which arguments the maps $\beta$ and $\gamma$ are applied to.

Given a symplectic vector space $V$ whose symplectic form has odd degree, we use these tensors to define a map
\[ \Phi_M:\NonComBV{V}\to\NonComBV{V\otimes M}. \]
This map is defined on each subspace $\gamma^g\nu^b\csym^m\left(\NCHamP{V}\right)\subset\NonComBV{V}$ by the commutative diagram
\[ \xymatrix{V^{*\otimes k_1}\otimes\cdots\otimes V^{*\otimes k_m} \ar[rrr]^-{-\otimes\mu^{g,b}_{k_1,\ldots,k_m}} \ar[dd] &&& (V^{*\otimes k_1}\otimes\cdots\otimes V^{*\otimes k_m})\otimes(M^{*\otimes k_1}\otimes\cdots\otimes M^{*\otimes k_m}) \ar@{=}[d] \\ &&& (V\otimes M)^{*\otimes k_1}\otimes\cdots\otimes (V\otimes M)^{*\otimes k_m} \ar[d] \\ \gamma^g\nu^b\csym^m\left(\NCHamP{V}\right) \ar[rrr]^{\Phi_M} &&& \gamma^g\nu^b\csym^m\left(\NCHamP{V\otimes M}\right) } \]
where the vertical maps are the canonical quotient maps, multiplied by the factor $\gamma^g\nu^b$.

The following result is a direct consequence of the axioms for an open topological field theory; cf. \cite{hamNCBVform}, Theorem 5.1.

\begin{theorem} \label{thm_OTFTdglamap}
The map $\Phi_M$ is a well-defined map of differential graded Lie algebras.
\end{theorem}

\subsection{Calculation for the Frobenius algebra of $N\times N$ matrices}\label{Phi map on matrices}

In this section we will calculate the open topological field theory tensors \eqref{eqn_OTFTtensor} for the Frobenius algebra $M:=\mat{N}{\gf}$ and hence the map $\Phi_M$ of Theorem \ref{thm_OTFTdglamap}. If $E_{ij}$ denotes the matrix whose $ij$th entry is $1$ and all of whose other entries are $0$ then we have the following elementary identities:
\begin{gather*}
E_{ij}E_{j'k} = \delta_{jj'}E_{ik}, \qquad A=\sum_{i,j=1}^N a_{ij} E_{ij}, \qquad \tr (E_{ij}E_{kl}) = \delta_{jk}\delta_{il}; \\
A E_{ij} = \sum_{k=1}^N a_{ki}E_{kj}, \qquad E_{ij} A = \sum_{k=1}^N a_{jk}E_{ik}. \\
\end{gather*}
The inverse inner product on $\mat{N}{\gf}$ is
\[ \innprod^{-1} = \sum_{i,j=1}^n E_{ij}\otimes E_{ji}. \]

We begin by computing the maps $\beta$ and $\gamma$ defined by \eqref{eqn_OTFTfreebdry} and \eqref{eqn_OTFTfreegenus} respectively:
\begin{align}
\label{eqn_calcOTFTfreebdry} \beta(A) &= \sum_{i,j=1}^N E_{ij}E_{ji} A = N\sum_{i=1}^N E_{ii} A = N\sum_{i,k=1}^N a_{ik} E_{ik} = NA, \\
\label{eqn_calcOTFTfreegenus} \gamma(A) &= \sum_{i,j,k,l=1}^N E_{ij}E_{kl}E_{ji}E_{lk} A = \sum_{i,j=1}^N E_{ij}E_{ji}E_{ji}E_{ij} A = \sum_{i=1}^N E_{ii} A = A.
\end{align}
Next we compute the tensor \eqref{eqn_OTFTnulltensor}:
\begin{multline*}
\mu^{0,0}_{k_1,\ldots,k_m}\left(A^{11},\ldots A^{1k_1};\ldots;A^{m1},\ldots,A^{mk_m}\right) = \\
\sum_{i_1,j_1;\ldots;i_m,j_m=1}^N\tr\Big(E_{i_m j_m}\cdots E_{i_1 j_1}\Big)\tr\Big(E_{j_1 i_1}A^{11}\cdots A^{1k_1}\cdots E_{j_m i_m}A^{m1}\cdots A^{mk_m}\Big).
\end{multline*}
For convenience, we introduce the notation $\mathbf{A}^r:=A^{r1}\cdots A^{rk_r}$. The entries of the matrix $\mathbf{A}^r$ will be denoted by $\alpha^r_{ij}$. We continue our calculation:
\begin{multline} \label{eqn_calcOTFTnulltensor}
\mu^{0,0}_{k_1,\ldots,k_m}\left(A^{11},\ldots A^{1k_1};\ldots;A^{m1},\ldots,A^{mk_m}\right) = \ldots \\
= \sum_{i_1,\ldots, i_m=1}^N\tr\Big(E_{i_m i_1}\mathbf{A}^1 E_{i_1 i_2}\mathbf{A}^2\cdots E_{i_{m-2} i_{m-1}}\mathbf{A}^{m-1} E_{i_{m-1} i_m}\mathbf{A}^m\Big), \\
= \sum_{i_1,l_1;\ldots;i_m,l_m=1}^N \alpha^1_{i_1 l_1}\alpha^2_{i_2 l_2}\cdots\alpha^{m-1}_{i_{m-1} l_{m-1}}\alpha^m_{i_m l_m}\tr\left(E_{i_m l_1}E_{i_1 l_2}\cdots E_{i_{m-2} l_{m-1}}E_{i_{m-1} l_m}\right), \\
= \sum_{i_1,\ldots,i_m=1}^N \alpha^1_{i_1 i_1}\alpha^2_{i_2 i_2}\cdots\alpha^{m-1}_{i_{m-1} i_{m-1}}\alpha^m_{i_m i_m} = \tr (\mathbf{A}^1)\cdots \tr (\mathbf{A}^m).
\end{multline}
Combining \eqref{eqn_calcOTFTfreebdry} and \eqref{eqn_calcOTFTfreegenus} with \eqref{eqn_calcOTFTnulltensor} concludes our calculation of the open topological field theory associated to $\mat{N}{\gf}$ and yields
\begin{equation} \label{eqn_calcOTFTtensor}
\mu^{g,b}_{k_1,\ldots,k_m}\left(A^{11},\ldots A^{1k_1};\ldots;A^{m1},\ldots,A^{mk_m}\right) = N^b\tr (A^{11}\cdots A^{1k_1})\cdots \tr (A^{m1}\cdots A^{mk_m}).
\end{equation}

We may now give a proof of Theorem \ref{thm_OTFTmatrixdglamap}.

\begin{proof}[Proof of Theorem \ref{thm_OTFTmatrixdglamap}]
It follows from Equation \eqref{eqn_calcOTFTtensor} that the map $\widehat{\Mor}_{\gamma,\nu}$ coincides with the map $\Phi_{\mat{N}{\gf}}$, which is a morphism of differential graded Lie algebras by Theorem \ref{thm_OTFTdglamap}.
\end{proof}

As a Corollary, we get a proof of Lemma \ref{lem_Moritamapbialgebras}.

\begin{proof}[Proof of Lemma \ref{lem_Moritamapbialgebras}]
The Lie algebra $\NCHam{V}$ sits inside $\NonComBV{V}$ as a Lie subalgebra. The restriction of the Lie algebra morphism $\widehat{\Mor}_{\gamma,\nu}$ to the Lie subalgebra $\NCHam{V}$ is the map
\[ \Mor:\NCHam{V}\to\NCHam{V\otimes\mat{N}{\gf}}, \]
which hence must also be a map of Lie algebras.

To prove that $\Mor$ respects the cobrackets we use the fact that $\widehat{\Mor}_{\gamma,\nu}$ must commute with the differentials on $\NonComBV{V}$ and $\NonComBV{V\otimes\mat{N}{\gf}}$. Applying this to $x\in\NCHam{V}\subset\NonComBV{V}$ we get,
\begin{displaymath}
\begin{split}
\nabla\Mor(x) &= (\nabla+\gamma\delta)\Mor(x) = (\nabla+\gamma\delta)\widehat{\Mor}_{\gamma,\nu}(x) \\
&= \widehat{\Mor}_{\gamma,\nu}(\nabla+\gamma\delta)(x) = \widehat{\Mor}_{\gamma,\nu}\nabla(x) = (\Mor\cotimes\Mor)\nabla(x).
\end{split}
\end{displaymath}
\end{proof}

\section{Random matrices and Hermitian matrix integrals} \label{sec_random}

Having introduced the basic mathematical objects, structures and theorems underlying our approach to studying large $N$ phenomena in the preceding sections of the paper, we are now ready to describe in this section some of the applications of this cohomological framework to Hermitian matrix integrals. We would like to emphasize here that the scope of our inquiries will be essentially limited to explaining the appearance of noncommutative geometry in the large $N$ limit of these matrix models and describing how the family of maps \eqref{eqn_LQTdglamaps} may be used to analyze the dependence of these models upon the rank $N$. For a suitable and particularly simple choice of cyclic $\ai$-algebra, the machinery of the previous section may be employed in the analysis and calculation of certain expectation values in these matrix models. In order to keep the current paper within reasonable limits, we intend that a full account of our results, including a description of the large $N$ asymptotic behavior of these quantities, will appear in a subsequent article.

\subsection{Expected values of multi-trace operators}

Let $\her{N}$ denote the real subspace of $\gl{N}{\mathbb{C}}$ consisting of Hermitian matrices. The trace pairing on $\gl{N}{\mathbb{C}}$ restricts to a positive definite symmetric bilinear form on $\her{N}$. We will be principally interested in the expected value of the multi-trace operator:
\begin{equation} \label{eqn_expvaltrace}
I_{i_1,i_2,\ldots,i_k}^N:=\frac{\int_{\her{N}}\tr (X^{i_1})\tr (X^{i_2})\cdots\tr (X^{i_k})e^{-\frac{1}{2}\tr (X^2)}\d X}{\int_{\her{N}}e^{-\frac{1}{2}\tr (X^2)}\d X}.
\end{equation}
We note that while the integrals in both the numerator and denominator of \eqref{eqn_expvaltrace} depend upon a linear identification of $\her{N}$ with $\mathbb{R}^{(n+1)n/2}$; the ratio, of course, does not.

\subsection{A cohomological approach through the noncommutative Batalin-Vilkovisky formalism}

We now define an extremely simple cyclic $\ai$-algebra which, upon plugging into the map \eqref{eqn quant lqt poly} below Theorem~\ref{thm hbar LQT} and setting $\hbar = 1$, leads to a description of the integrals \eqref{eqn_expvaltrace} appearing in random matrix theory. Consider the two-dimensional complex vector space that is freely spanned by two generators $a$ and $b$ of degrees 1 and 2 respectively. This has a symmetric pairing and a compatible differential;
\begin{equation} \label{eqn_twodimalg}
\langle a,b \rangle = 1, \qquad da=b.
\end{equation}
We may regard this complex as a cyclic $\ai$-algebra $\twodim$ in which the multiplication and all higher homotopies vanish.

Throughout the remainder of this section, we use $\NCBVGauss{\Sigma\twodim}$ and $\ComBVGauss{\Sigma\gl{N}{\twodim}}$ to indicate the complexes $\NCBVpoly{\Sigma\twodim}$ and $\ComBVpoly{\Sigma\twodim}$ modulo the ideal~$(\hbar -1)$. Here we emphasize that we include the differential $d$ from the cyclic $\ai$-algebra $\twodim$, cf. Remark \ref{rem_BVdiffdeform}. We may realize these concretely as BV-algebras as follows. $\ComBVGauss{\Sigma\gl{N}{\twodim}}$ has underlying shifted Poisson algebra $\Poispoly{\Sigma\gl{N}{\twodim}}$ and is equipped with the differential $(d+\Delta)$. $\NCBVGauss{\Sigma\twodim}$ has underlying shifted Poisson algebra $S(\NCHampoly{\Sigma\twodim})$ and is equipped with the differential $(d+\delta+\nabla)$.

The following is a simple and direct consequence of Lemma \ref{lem_shiftpoismap} and Theorem \ref{thm hbar LQT}.

\begin{prop} \label{prop_polyBVmap}
For every positive integer $N$, the map
\begin{equation} \label{eqn_polyBVmap}
\NCBVGauss{\Sigma\twodim} \longrightarrow \ComBVGauss{\Sigma\gl{N}{\twodim}}
\end{equation}
obtained by specializing the map \eqref{eqn quant lqt poly} to $\hbar = 1$ is a map of BV-algebras.
\end{prop}

This abstract statement will be unwound to something extremely concrete:
\begin{itemize}
\item for a polynomial $p(x) = \sum a_n x^n$, viewed as a degree zero cocycle in $\NCBVGauss{\Sigma\twodim}$ where $x$ is dual to the generator $a$, its image is the functional $p(X) = \sum a_n \tr(X^n)$ on $X$ a $N \times N$-matrix, and
\item at the level of cohomology, the cohomology class $[p(x)]$ is a polynomial in $\nu$ whose value at $\nu = N$ is the expected value of the operator~$p(X)$.
\end{itemize}
Much of the rest of this section is devoted to demonstrating these claims.

Consider the real subspace $\her{N}$ of $\Sigma\gl{N}{\twodim}$ consisting of Hermitian matrices, which sits in the degree zero part of
\[ \Sigma\gl{N}{\twodim} = \Sigma\twodim\otimes\gl{N}{\mathbb{C}} \]
and corresponds to the generator $a$ of $\Sigma\twodim$ (which has degree zero after the shift). Restricting a polynomial superfunction from $\ComBVGauss{\Sigma\gl{N}{\twodim}}$ to this subspace yields a complex-valued function on $\her{N}$, which we may integrate against the Gaussian measure to define the expected value
\begin{equation} \label{eqn_expvalpoly}
\langle - \rangle: \ComBVGauss{\Sigma\gl{N}{\twodim}} \longrightarrow \mathbb{C}
\end{equation}
by
\[ \langle f \rangle:= \frac{\int_{\her{N}}f(X)e^{-\frac{1}{2}\tr (X^2)}\d X}{\int_{\her{N}}e^{-\frac{1}{2}\tr (X^2)}\d X}. \]

\begin{prop} \label{prop_expval}
The map \eqref{eqn_expvalpoly} is a quasi-isomorphism of complexes whose one-sided inverse is the inclusion of $\mathbb{C}$ inside $\ComBVGauss{\Sigma\gl{N}{\twodim}}$ as the constant polynomials.
\end{prop}

\begin{proof}
We begin by explaining the story over the ground field $\gf:=\mathbb{R}$. We may write the complex vector space $\twodim$ as the complexification of a real vector space $\twodim_{\mathbb{R}}$ that is spanned by the same generators $a$ and $b$ and which is equipped with the same cyclic infinity-structure given by \eqref{eqn_twodimalg}. The induced cyclic $\li$-structure on
\[ \her{N}^{\twodim}:=\twodim_{\mathbb{R}}\underset{\mathbb{R}}{\otimes}\her{N}, \]
in which $\her{N}$ is equipped with the trace pairing, is represented by a quadratic monomial
\[ \sigma:=\tilde{d}\in\ComBVGauss{\Sigma\her{N}^{\twodim}}, \]
cf. Section \ref{sec_reltocyclicstruct}. If we identify the space $\her{N}$ of Hermitian matrices  with the subspace of
\[ \Sigma\her{N}^{\twodim}=\Sigma\twodim_{\mathbb{R}}\underset{\mathbb{R}}{\otimes}\her{N} \]
that sits in degree zero and corresponds to the generator $a$ of $\Sigma\twodim_{\mathbb{R}}$, then on this subspace
\[ \sigma(X)=\frac{1}{2}\tr (X^2), \quad X\in\her{N}. \]

Consider the map
\begin{equation} \label{eqn_expvalpolyreal}
\ComBVGauss{\Sigma\her{N}^{\twodim}} \longrightarrow \mathbb{R}, \qquad f \mapsto \frac{\int_{\her{N}}fe^{-\sigma}}{\int_{\her{N}}e^{-\sigma}};
\end{equation}
where we have again identified $\her{N}$ with the subspace of $\Sigma\her{N}^{\twodim}$ that sits in degree zero. We must show that \eqref{eqn_expvalpolyreal} is a cocycle. However, note that
\begin{equation} \label{eqn_quadraticQME}
\Delta e^{-\sigma}=(\{\sigma,\sigma\}-\Delta\sigma)e^{-\sigma}=0,
\end{equation}
as $\{\sigma,\sigma\}$ vanishes due to the $\li$-constraint and $\Delta\sigma$ vanishes as $\sigma$ is an even quadratic form. Therefore,
\begin{displaymath}
\begin{split}
\int_{\her{N}} (\Delta f+d^*f)e^{-\sigma} &= \int_{\her{N}} (\Delta f-\{\sigma,f\})e^{-\sigma}, \\
&= \int_{\her{N}}\Delta(fe^{-\sigma})=0;
\end{split}
\end{displaymath}
where we have used \eqref{eqn_BVidentity}, \eqref{eqn_CEbracketdiff} and \eqref{eqn_quadraticQME}. The last integral vanishes as $\her{N}$ is a Lagrangian subspace of $\Sigma\her{N}^{\twodim}$ and so completely routine results from the Batalin-Vilkovisky formalism apply here to show that the integral is indeed zero; cf. for instance Corollary 5.17 of \cite{graphBV}, which contains a standard treatment of these basic results.

We claim in fact that \eqref{eqn_expvalpolyreal} is a quasi-isomorphism. Consider the ascending filtration of the complex $\ComBVGauss{\Sigma\her{N}^{\twodim}}$,
\[ F_p\ComBVGauss{\Sigma\her{N}^{\twodim}}:=\bigoplus_{i=0}^p\left[\left(\left(\Sigma\her{N}^{\twodim}\right)^*\right)^{\otimes i}\right]_{\symg{i}}. \]
This filtration is bounded below and exhaustive and hence the associated spectral sequence converges. Since the complex $\twodim_{\mathbb{R}}$ is acyclic, it is apparent that \eqref{eqn_expvalpolyreal} induces an isomorphism on the first page of this spectral sequence.

Now, returning to the main story over the ground field $\gf:=\mathbb{C}$, we use the simple well-known fact that $\gl{N}{\mathbb{C}}$ is the complexification of $\her{N}$;
\[ \Sigma\gl{N}{\twodim} = \mathbb{C}\underset{\mathbb{R}}{\otimes}\Sigma\her{N}^{\twodim}. \]
From this it is apparent that the map \eqref{eqn_expvalpoly} is the complexification of the map \eqref{eqn_expvalpolyreal}. The result now follows.
\end{proof}

To provide a noncommutative analogue of Proposition \ref{prop_expval}, consider the following. The polynomial ring $\mathbb{C}[\nu]$ appears as a summand in,
\[ \NCBVGauss{\Sigma\twodim} = \mathbb{C}[\nu]\otimes \sym(\NCHamPpoly{\Sigma\twodim}) = \mathbb{C}[\nu]\oplus\left[\mathbb{C}[\nu]\otimes \symp(\NCHamPpoly{\Sigma\twodim})\right]. \]

\begin{prop} \label{prop_NCBVqiso}
The canonical inclusion
\[ \mathbb{C}[\nu] \longrightarrow \NCBVGauss{\Sigma\twodim} \]
is a quasi-isomorphism of complexes.
\end{prop}

\begin{proof}
We may simply repeat the same spectral sequence argument that was applied in Proposition \ref{prop_expval} above.
\end{proof}

In summary, for every positive integer $N$ we have the commutative diagram
\begin{equation} \label{eqn_expvaldiagram}
\xymatrix{\mathbb{C}[\nu] \ar[r] \ar[d]_{\nu=N} & \NCBVGauss{\Sigma\twodim} \ar[d] \\ \mathbb{C} \ar@<0.5ex>@{-^{>}}[r] & \ComBVGauss{\Sigma\gl{N}{\twodim}} \ar@<0.5ex>@{-^{>}}[l]^-{\langle f \rangle \mapsfrom f} }
\end{equation}
in which the horizontal arrows are quasi-isomorphisms and the vertical arrows are given by~\eqref{eqn_polyBVmap}.

\subsection{Examples of calculations with $\twodim$}

Now we are ready to explain how the cyclic $\ai$-algebra $\twodim$ may be used to describe and calculate the multi-trace expectation values \eqref{eqn_expvaltrace}.

First we introduce some notation. Recall that $\Sigma\twodim$ is a symplectic vector space whose symplectic form is given by \eqref{eqn_sympform},
\[ \langle b,a \rangle = 1 = -\langle a,b \rangle. \]
If $a^*$ and $b^*$ denote the elements of $\Sigma\twodim^*$ that are dual to $a$ and $b$ respectively then define,
\[ x:=a^* \quad\text{and}\quad \xi:=-b^*. \]
According to \eqref{eqn_invinnprod} we have
\[ \{x,\xi\} = 1 = \{\xi,x\} \quad\text{and}\quad d^*\xi = -x. \]
Note that $x$ has degree zero and $\xi$ has degree minus-one.

We will denote by $x^i$ the corresponding monomial in $\NCHampoly{\Sigma\twodim}$ and by
\begin{equation} \label{eqn_ncstring}
(x^{i_1})(x^{i_2})\cdots(x^{i_k})\in\left[\NCHampoly{\Sigma\twodim}^{\otimes k}\right]_{\symg{k}}
\end{equation}
the corresponding element in $\NCBVGauss{\Sigma\twodim}=\sym(\NCHampoly{\Sigma\twodim})$. Now \eqref{eqn_ncstring} is a cocycle in $\NCBVGauss{\Sigma\twodim}$ and so by Proposition \ref{prop_NCBVqiso}, its cohomology class is represented by a unique polynomial
\[ p_{i_1,i_2,\ldots,i_k}(\nu)\in\mathbb{C}[\nu]. \]

\begin{prop} \label{prop_polycorrelationval}
For every positive integer $N$,
\[ p_{i_1,i_2,\ldots,i_k}(N) = I_{i_1,i_2,\ldots,i_k}^N = \frac{\int_{\her{N}}\tr (X^{i_1})\tr (X^{i_2})\cdots\tr (X^{i_k})e^{-\frac{1}{2}\tr (X^2)}\d X}{\int_{\her{N}}e^{-\frac{1}{2}\tr (X^2)}\d X}. \]
\end{prop}

\begin{proof}
This follows from Proposition \ref{prop_polyBVmap}, Proposition \ref{prop_expval} and Proposition \ref{prop_NCBVqiso} as the image of $(x^i)$ under the right-hand map of \eqref{eqn_expvaldiagram} is the function sending $X$ to $\tr (X^i)$ for $X\in\gl{N}{\mathbb{C}}$.
\end{proof}

The preceding proposition is the first clue that the homological algebra of the complex $\NCBVGauss{\Sigma\twodim}$ which arises through the noncommutative Batalin-Vilkovisky formalism captures features of random matrix theory.

We now demonstrate in some simple cases how these results may be applied to compute single and multi-trace expectation values. For instance,
\[ x^2 = -d^*(x\xi) \approx \Delta(x\xi) = \nabla(x\xi) = \nu^2 \]
and hence $p_2(\nu)=\nu^2$.

Before proceeding further we stop to note that $p_{i_1,i_2,\ldots,i_k}(\nu)$ is equal to zero whenever the sum of the indices $i_j$ is odd. This is because $\Delta$ is an order two operator and $d^*$ is an order zero operator.

Continuing we have,
\[ (x)(x) = -d^*((x)(\xi)) \approx \Delta((x)(\xi)) = \nu \]
and hence $p_{1,1}(\nu)=\nu$.

Using our preceding calculations we may compute $p_4(\nu)$ as follows:
\begin{displaymath}
\begin{split}
x^4 &= -d^*(x^3\xi) \approx \Delta(x^3\xi) = \nabla(x^3\xi), \\
&= 2\nu x^2 + (x)(x) \approx 2\nu^3 + \nu.
\end{split}
\end{displaymath}
Hence $p_4(\nu)=2\nu^3+\nu$.

Moving on we calculate
\begin{displaymath}
\begin{split}
(x)(x^3) &= -d^*((\xi)(x^3)) \approx \Delta((\xi)(x^3)) = 3(x^2) \approx 3\nu^2, \\
(x^2)(x^2) & = -d^*((x\xi)(x^2)) \approx \Delta((x\xi)(x^2)) = \nabla(x\xi)(x^2) + \{(x\xi),(x^2)\}, \\
&= \nu^2(x^2) + 2(x^2) \approx \nu^2(\nu^2+2), \\
(x^6) &= -d^*(x^5\xi) \approx \Delta(x^5\xi) = \nabla(x^5\xi), \\
&= 2\nu(x^4) + 2(x)(x^3) + (x^2)(x^2), \\
&\approx 2\nu^2(2\nu^2+1) + 6\nu^2 + \nu^2(\nu^2+2) = \nu^2(5\nu^2+10);
\end{split}
\end{displaymath}
and therefore
\begin{displaymath}
\begin{split}
p_{1,3}(\nu) &= 3\nu^2, \\
p_{2,2}(\nu) &= \nu^4 + 2\nu^2, \\
p_6(\nu) &= 5\nu^4 + 10\nu^2.
\end{split}
\end{displaymath}

Using the above method, the reader may  similarly verify
\begin{displaymath}
\begin{split}
p_{1,5}(\nu) &= 10\nu^3 + 5\nu, \\
p_{2,4}(\nu) &= 2\nu^5 + 9\nu^3 + 4\nu, \\
p_{3,3}(\nu) &= 12\nu^3 + 3\nu, \\
p_8(\nu) &= 14\nu^5 + 70\nu^3 + 21\nu, \\
p_{1,7}(\nu) &= 35\nu^4 + 70\nu^2, \\
p_{2,6}(\nu) &= 5\nu^6 + 40\nu^4 + 60\nu^2, \\
p_{3,5}(\nu) &= 45\nu^4 + 60\nu^2, \\
p_{4,4}(\nu) &= 4\nu^6 + 40\nu^4 + 61\nu^2, \\
p_{10}(\nu) &= 42\nu^6 + 420\nu^4 + 483\nu^2.
\end{split}
\end{displaymath}
We emphasize that, in principle, any single or multi-trace expectation value may be computed using these techniques. Here is a simple relation that may be deduced using this method.

\begin{prop}
For every positive integer $N$,
\[ \sum_{i=1}^{2k-1} I^N_{i,2k-i} = I^N_{2k+2} - 2NI^N_{2k}. \]
\end{prop}

\begin{proof}
We compute
\begin{displaymath}
\begin{split}
x^{2k+2} &= -d^*(x^{2k+1}\xi) \approx \Delta(x^{2k+1}\xi) = \nabla(x^{2k+1}\xi), \\
&= \sum_{i=1}^{2k+1}(x^{i-1})(x^{2k+1-i}), \\
&= 2\nu(x^{2k}) + \sum_{i=1}^{2k-1}(x^i)(x^{2k-i}).
\end{split}
\end{displaymath}
\end{proof}

And finally, we can give a very simple description of these expectation values in the following special case.

\begin{prop}
For every positive integer $N$,
\[ I^N_{\underbrace{1, 1, \dots, 1}_{2n \text{~terms}}} = N^n (2n-1)!!\]
\end{prop}

\begin{proof}
We compute
\begin{displaymath}
\begin{split}
(x)^{2n} &= -d^*((x)^{2n-1} (\xi)) \approx \Delta((x)^{2n-1}(\xi)) = \delta((x)^{2n-1}(\xi)) \\
&= (2n-1)\nu(x)^{2n-2} \\
&\approx (2n-1)(2n-3)\nu^2(x)^{2n-4}\\
&\approx (2n-1)!!\nu^n
\end{split}
\end{displaymath}
\end{proof}

As mentioned earlier, a more complete discussion of this subject and the large $N$ asymptotics of the correlation functions \eqref{eqn_expvaltrace} will be provided in a subsequent article. We note however, that the Harer-Zagier recurrence relation is encoded in our framework as follows.

\begin{theorem}
The polynomials $p_{2k}$ satisfy the Harer-Zagier recurrence relation:
\begin{equation} \label{eqn_HZrecpoly}
(k+1) p_{2k}(\nu) =  (4k-2) \nu p_{2k-2}(\nu) + (k-1)(2k-1)(2k-3)p_{2k-4}(\nu).
\end{equation}
\end{theorem}

\begin{proof}
By Proposition \ref{prop_polycorrelationval}, evaluating \eqref{eqn_HZrecpoly} at $\nu=N$ yields \eqref{eqn_HZrec}. Since this equation holds for all $N$ and a nontrivial polynomial may have only finitely many roots, the theorem follows.
\end{proof}

As noted earlier, this recurrence relation implies Wigner's semicircle law as well as the closed formula \eqref{eqn_HZformula} for $p_{2k}$ of Harer-Zagier.

\section{Obstruction theory for the quantum master equation} \label{sec_obsthy}

Having addressed the relevance of our framework to random matrix theory in the previous section, we now move on to describe how it may be used to analyze the problem of quantization within the Batalin-Vilkovisky formalism in the large $N$ limit. Strictly speaking, this will be the subject of Section \ref{sec_largeN}, as in the current section we will be occupied with the construction of a framework in which to describe most generally the problem of quantization in the Batalin-Vilkovisky formalism. Here we will explain how the examples from Section \ref{sec_commBV} and \ref{sec_relcomnoncom} fit into this general framework. This work will be necessary for what follows in the final section.

Following the introduction of our general framework, we prove a theorem that allows us to deduce a one-to-one correspondence between moduli spaces of quantizations, provided that a mapping induces an isomorphism between the cohomology theories controlling the quantization process. Experts will probably substantively recognize this theorem as the well-known theorem of Goldman-Milson, Theorem 2.4 of \cite{goldmill}; which also appears elsewhere in the literature in various guises, see e.g. \cite{getzliethy} and Theorem 2.1 of \cite{getzdarboux}. Although the method of proof presented here is largely the same, differences arise; for example our differential graded Lie algebras are not nilpotent and our notion of a filtration differs slightly from the standard one. This prevents us from merely citing the results from the above sources.

\subsection{General framework}

\subsubsection{Definitions}

Our starting point is a differential $\mathbb{Z}/2\mathbb{Z}$-graded Lie algebra
\[ \left(\mathfrak{g},d,\{-,-\}\right) \]
whose Lie bracket $\{-,-\}$ has \emph{odd} degree. We require this to come equipped with a filtration of the type which we now specify.

\begin{defi}
A \emph{strong filtration} on $\mathfrak{g}$ is a complete Hausdorff decreasing filtration
\[
F_0\mathfrak{g} \supset F_1\mathfrak{g} \supset \cdots
\]
satisfying
\begin{flalign} \label{eqn_BVfiltration}
\begin{array}{ll}
F_0\mathfrak{g} = \mathfrak{g}. \\
\left\{F_p\mathfrak{g},F_q\mathfrak{g}\right\}\subset F_{p+q}\mathfrak{g}; & \quad p,q\geq 0. \\
d(F_p\mathfrak{g})\subset F_{p+1}\mathfrak{g}, & \quad p\geq 0.
\end{array}
&&
\end{flalign}
\end{defi}

One consequence of this definition is that the $F_p\mathfrak{g}$ are differential graded ideals. We note that the above definition differs from the customary notion of a filtered differential graded Lie algebra in that the differential is now required to increase the filtration degree. We assume, for the rest of this section, that our differential graded Lie algebras come equipped with a strong filtration.

\subsubsection{Examples}

In our paper we have already encountered several examples.

\begin{example}
Given a symplectic vector space $V$ with an odd symplectic form, the differential graded Lie algebra
\[ \ComBV{V} = \left(\gf[[\hbar]]\cotimes\csym V^*,\{-,-\},\hbar\Delta\right) \]
admits the strong filtration
\begin{equation} \label{eqn_commfilterexample}
F_p\ComBV{V} := \hbar^p\ComBV{V}
\end{equation}
by powers of $\hbar$.
\end{example}

\begin{example} \label{exm_noncomBVfiltration}
Given a symplectic vector space $V$ whose symplectic form is odd,
consider the differential graded Lie algebra $\NonComBV{V}$ whose underlying space is
\[ \gf[[\gamma]]\cotimes\csymp(\NCHam{V}) \subset \gf[[\gamma]]\cotimes\csym(\NCHam{V}) = \gf[[\gamma,\nu]]\cotimes\csym(\NCHamP{V}). \]
Define the order of $\gamma$ to be $2$, the order of $\nu$ to be $1$ and the order of an element in the $k$-fold symmetric power of $\NCHamP{V}$ to be $(k-1)$ so that an element
\[ \gamma^i\nu^j h_1 h_2\cdots h_k\in\NonComBV{V} \]
has order $(2i+j+k-1)$, where $h_r\in\NCHamP{V}$.
This notion of order determines a descending filtration
\[ F_p\NonComBV{V} := \text{`everything of order $\geq p$'} \]
that is strong.
\end{example}

\begin{example}
Given a cyclic $\ai$-algebra $A$ whose inner product has odd degree, consider the differential graded Lie algebra
\[ \ComBV{\Sigma\gl{N}{A}}^{\gl{N}{\gf}} \]
formed by taking the $\gl{N}{\gf}$-invariants of $\ComBV{\Sigma\gl{N}{A}}$.
This is a strongly filtered differential graded Lie algebra where the filtration is the one induced by~\eqref{eqn_commfilterexample}.
\end{example}

\subsection{Maurer-Cartan elements}

\subsubsection{Definitions}

Consider the Maurer-Cartan functor that assigns to a strongly filtered differential graded Lie algebra $\mathfrak{g}$, its Maurer-Cartan set
\[ \MCset{\mathfrak{g}}:=\left\{x\in\mathfrak{g}^0:dx+\frac{1}{2}[x,x]=0\right\}. \]
The subspace $F_1\mathfrak{g}^1$ acts on $\MCset{\mathfrak{g}}$ by the formula
\begin{equation} \label{eqn_MCact}
\exp(y)\cdot x:= x + \sum_{i=0}^\infty \frac{1}{(i+1)!}\{y,-\}^i\left(dy+\{y,x\}\right); \quad y\in F_1\mathfrak{g}^1,x\in\MCset{\mathfrak{g}},
\end{equation}
where $\{y,-\}:a\mapsto\{y,a\}$ denotes the adjoint map. We insist that $y$ lie in $F_1\mathfrak{g}^1$ in order for the above sum to converge, where we rely on the fact that the filtration is complete.

\begin{rem}
Note that since the Lie bracket on $\mathfrak{g}$ is \emph{odd}, the Maurer-Cartan elements live in degree zero rather than one. This particular feature also explains other minor variations in, for instance, signs in formulas such as \eqref{eqn_MCact} above, that the reader well-acquainted with Maurer-Cartan spaces might pick up on.
\end{rem}

The twisted action defined by \eqref{eqn_MCact} corresponds to the standard adjoint action on the semi-direct product of $\mathfrak{g}$ with (the shift of) $\gf$, where $\gf$ acts on $\mathfrak{g}$ using the differential $d$, cf. \cite{nijenhuisrichardsondeformations} Section 3. As such, standard formulae apply to this action, including the Baker-Campbell-Hausdorff formula. Define
\begin{equation} \label{eqn_BCHformula}
z \star y := \sum_{k=1}^{\infty} \frac{(-1)^{k-1}}{k}\sum_{\begin{subarray}{c} i_1 + j_1 >0 \\ \vdots \\ i_k+j_k > 0 \end{subarray}}\frac{\{z^{i_1}y^{j_1} \cdots z^{i_k}y^{j_k}\}}{\left(\sum_{l=1}^k (i_l+j_l) \right)\left(\prod_{m=1}^k (i_m! j_m!) \right)}; \quad y,z\in F_1\mathfrak{g}^1;
\end{equation}
where
\[ \{z^{i_1}y^{j_1} \cdots z^{i_k}y^{j_k}\} :=\{\underbrace{z,\{z,\{z,\cdots}_{i_1 \text{ times}}\{\underbrace{y,\{y,\{y,\cdots}_{j_1 \text{ times}}\{\cdots\}\}\cdots\} \]
is a Lie monomial in $y$ and $z$. Then $F_1\mathfrak{g}^1$ forms a group under $\star$ with the identity given by $0$ and the inverse of an element $y$ in $F_1\mathfrak{g}^1$ given by $-y$. Note that the sum in \eqref{eqn_BCHformula} converges as the filtration is complete.
In fact, \eqref{eqn_MCact} is a group action,
\[ \exp(z)\cdot\left(\exp(y)\cdot x\right) = \exp(z\star y)\cdot x; \quad y,z\in F_1\mathfrak{g}^1,x\in\MCset{\mathfrak{g}}, \]
by the Baker-Campbell-Hausdorff Theorem.
We will view the quotient of the set $\MCset{\mathfrak{g}}$ by the action of the group $F_1\mathfrak{g}^1$ as a \emph{moduli space}.

In the Batalin-Vilkovisky approach to quantization, one begins at the classical level,
which means in the shifted Lie algebra.
This notion admits a nice formulation in the strongly filtered setting.

\begin{defi}
A solution to the {\em classical master equation} is an element
\[ x_0\in\MCset{\mathfrak{g}/F_1\mathfrak{g}}. \]
\end{defi}

\begin{rem}
In our examples, the quotient algebras $\mathfrak{g}/F_1\mathfrak{g}$ are exactly the ``classical'' aspect of the commutative and noncommutative BV formalisms, cf. Section \ref{sec_exampleCME}.
\end{rem}

Note that by \eqref{eqn_BVfiltration} the differential on $\mathfrak{g}/F_1\mathfrak{g}$ is zero and hence the Maurer-Cartan constraint in this case is simply $\{x_0,x_0\}=0$. Now the quotient map from $\mathfrak{g}$ to $\mathfrak{g}/F_1\mathfrak{g}$ induces a map
\[ \MCset{\mathfrak{g}}\to\MCset{\mathfrak{g}/F_1\mathfrak{g}} \]
which we view as a kind of \emph{dequantization} map.

\begin{defi}
Given a solution to the classical master equation $x_0\in\MCset{\mathfrak{g}/F_1\mathfrak{g}}$, we use $\MCfiberset{\mathfrak{g}}{x_0}$ to denote the fiber of this dequantization map. An element of this fiber will be called a \emph{quantization} of $x_0$. Note that $F_1\mathfrak{g}^1$ acts on this fiber by \eqref{eqn_MCact}.
The \emph{moduli space of quantizations} of $x_0$ is then the quotient of this fiber by the action of $F_1\mathfrak{g}^1$. We will denote this moduli space  of quantizations by $\MCfibermod{\mathfrak{g}}{x_0}$.
\end{defi}

We can refine this set of quantizations to a groupoid of quantizations, as follows.
Using \eqref{eqn_MCact}, we define the notion of a morphism between two quantizations of a solution $x_0$ to the classical master equation.

\begin{defi}
Given two quantizations $x,x'\in\MCfiberset{\mathfrak{g}}{x_0}$, a \emph{morphism} from $x$ to $x'$ is an element $y\in F_1\mathfrak{g}^1$ satisfying
\[ x' = \exp(y)\cdot x. \]
We denote the set of all such morphisms by $\Morph{x}{x'}$.
\end{defi}

Note that given any $x\in\MCset{\mathfrak{g}}$ and any element $\eta\in F_1\mathfrak{g}^0$, the element
\[ d\eta+\{\eta,x\}\in F_1\mathfrak{g}^1, \]
will be a morphism from $x$ to itself.
This assertion follows from the well-known fact that any Maurer-Cartan element $x$ can be used to twist the differential $d$ to a new differential:
\[ z\mapsto dz + \{z,x\}. \]
This leads to the following definition.

\begin{defi}
Two morphisms $y,y'\in F_1\mathfrak{g}^1$ from a quantization $x$ to a quantization $x'$ of $x_0$ will be called \emph{homotopy equivalent} if
\[ y' = y \star (d\eta +\{\eta,x\}) \]
for some $\eta\in F_1\mathfrak{g}^0$. We denote the set of homotopy equivalence classes of such morphisms by $\MorphH{x}{x'}$.
\end{defi}

\subsubsection{Examples of solutions to the classical master equation} \label{sec_exampleCME}

We describe the solutions to the classical master equation for our main examples.

\begin{example}
For $\ComBV{V}$ with its strong filtration by powers of $\hbar$,
we have
\[ \ComBV{V}/F_1\ComBV{V} = \left(\csym V^*,\{-,-\}\right). \]
In particular, it follows from the discussion in Section \ref{sec_reltocyclicstruct} that any cyclic $\li$-structure $l$ on the space $\Sigma^{-1}V$ yields a solution $x_0:=\tilde{l}$ to the classical master equation for $\ComBV{V}$.
More generally, the set of all solutions of $\ComBV{V}$ will consist of precisely all of the so-called \emph{curved} cyclic $\li$-structures; see \cite{linftwist, markldefthy} for a definition.
\end{example}

\begin{example}
For $\NonComBV{V}$ with its strong filtration by order,
we have
\[ \NonComBV{V}/F_1\NonComBV{V} = \left(\NCHam{V},\{-,-\}\right). \]
Again, it follows from the discussion in Section \ref{sec_BVNonComGeom} that any cyclic $\ai$-structure $m$ on the space $A:=\Sigma^{-1}V$ yields a solution $x_0:=\tilde{m}$ to the classical master equation for $\NonComBV{V}$. Likewise, the set of all solutions consists of precisely all the \emph{curved} cyclic $\ai$-structures.
\end{example}

\begin{example} \label{exm_comglnfiltration}
For $\ComBV{\Sigma\gl{N}{A}}^{\gl{N}{\gf}}$,
we have
\[ \ComBV{\Sigma\gl{N}{A}}^{\gl{N}{\gf}}/F_1\ComBV{\Sigma\gl{N}{A}}^{\gl{N}{\gf}} = \left(\left[\csym\Sigma\gl{N}{A}^*\right]^{\gl{N}{\gf}},\{-,-\}\right). \]
The cyclic commutator $\li$-structure $l_N$ on $\gl{N}{A}$ arising from the cyclic $\ai$-structure on $A$ is $\gl{N}{\gf}$-invariant and hence gives rise to a solution,
\[ x_0:=\tilde{l}_N\in\MCset{\left[\csym\Sigma\gl{N}{A}^*\right]^{\gl{N}{\gf}}} \]
to the classical master equation for $\ComBV{\Sigma\gl{N}{A}}^{\gl{N}{\gf}}$.
\end{example}

\subsection{Obstruction theory}

The process of building a quantization of a solution to the classical master equation is a step-by-step process controlled by a cohomology theory similar to the process of deforming other algebraic structures, cf. \cite{gerstdeform}. In this section we describe the general obstruction theory and cohomology theory controlling this process and its morphisms. Since this material is ultimately quite standard and well-known in other contexts, our intention here will be to be quite brief, sketching the details of proofs and providing references to the reader interested in a fuller treatment.

\subsubsection{Finite level structures}

Again, we start with a differential graded Lie algebra $\mathfrak{g}$ equipped with a strong filtration. This filtration provides a (finite) strong filtration on the quotient $\mathfrak{g}/F_{m+1}\mathfrak{g}$ and this leads to the notion of a \emph{finite level structure}.

\begin{defi}
For $m \geq 0$, an element
\[ x_m\in\MCset{\mathfrak{g}/F_{m+1}\mathfrak{g}}\]
is called a \emph{level $m$ structure}.
\end{defi}

Note that under this definition, a level~$0$ structure is precisely a solution to the classical master equation, as described in the preceding section. The natural quotient map induces a canonical map
\begin{equation} \label{eqn_leveldownmap}
\MCset{\mathfrak{g}/F_{m+2}\mathfrak{g}}\to\MCset{\mathfrak{g}/F_{m+1}\mathfrak{g}},
\end{equation}
from level $(m+1)$ structures to level $m$ structures. The first step is to characterize the nonempty fibers of this map in terms of a certain cohomology theory. First note that if $x_0\in\mathfrak{g}/F_1\mathfrak{g}$ is a level~$0$ structure then the adjoint map; \[ \{x_0,-\}: F_k\mathfrak{g}/F_{k+1}\mathfrak{g} \to F_k\mathfrak{g}/F_{k+1}\mathfrak{g},\quad k\geq 0; \]
is a differential.

\begin{prop} \label{prop_obstructionstruct}
There is a well-defined obstruction map on the moduli space of level $m$ quantizations of~$x_0$:
\begin{displaymath}
\begin{array}{ccc}
\MCfibermod{\mathfrak{g}/F_{m+1}\mathfrak{g}}{x_0} & \to & H^1\left(F_{m+1}\mathfrak{g}/F_{m+2}\mathfrak{g},\{x_0,-\}\right), \\
x_m & \mapsto & \Obs{x_m}:= dx_m+\frac{1}{2}\{x_m,x_m\}.
\end{array}
\end{displaymath}
A level $m$ structure $x_m\in\MCfiberset{\mathfrak{g}/F_{m+1}\mathfrak{g}}{x_0}$ may be lifted to a level $(m+1)$ structure under the map \eqref{eqn_leveldownmap} if and only if the cohomology class $\Obs{x_m}$ vanishes.
\end{prop}

\begin{proof}
Picking a choice of representative $\tilde{x}_m\in\mathfrak{g}^0$ for a level $m$ structure $x_m\in\MCfiberset{\mathfrak{g}/F_{m+1}\mathfrak{g}}{x_0}$ we have tautologically that
\[ d\tilde{x}_m+\frac{1}{2}\{\tilde{x}_m,\tilde{x}_m\}\in F_{m+1}\mathfrak{g}. \]
Furthermore, replacing $\tilde{x}_m$ with a different choice of representative only changes the above expression by a $\{x_0,-\}$-coboundary, after killing terms in $F_{m+2}\mathfrak{g}$ using \eqref{eqn_BVfiltration}. Additionally, if we act on $\tilde{x}_m$ using \eqref{eqn_MCact}, then since the adjoint action is a morphism of Lie algebras, the above expression changes through the application of the operator $\exp(\{y,-\})$, which only contributes terms coming from $F_{m+2}\mathfrak{g}$. The only step that requires any work is proving that the above expression is a $\{x_0,-\}$-cocycle. However, this argument proceeds -- mutatis mutandis -- exactly according to Theorem 5.1 of \cite{hamQME}, but cf. also the primary Proposition 2 in Chapter I of \cite{gerstdeform}.

Given $\xi_{m+1}\in F_{m+1}\mathfrak{g}^0$ consider the lift $\tilde{x}_{m+1}:=\tilde{x}_m + \xi_{m+1}$ of $\tilde{x}_m$, then
\begin{equation} \label{eqn_liftQME}
d\tilde{x}_{m+1} + \frac{1}{2}\{\tilde{x}_{m+1},\tilde{x}_{m+1}\} = d\tilde{x}_m + \frac{1}{2}\{\tilde{x}_m,\tilde{x}_m\} + \{x_0,\xi_{m+1}\} \quad\mod F_{m+2}\mathfrak{g}.
\end{equation}
Hence to lift a level $m$ structure, the obstruction must vanish.
\end{proof}

Now suppose that we are given a level $m$ structure $x_m\in\MCfiberset{\mathfrak{g}/F_{m+1}\mathfrak{g}}{x_0}$ and let
\[ \MCfiberset{\mathfrak{g}/F_{m+2}\mathfrak{g}}{x_m} \]
denote the fiber of the map \eqref{eqn_leveldownmap} over the point $x_m$. Then $F_{m+1}\mathfrak{g}^1/F_{m+2}\mathfrak{g}^1$ acts on this fiber via \eqref{eqn_MCact} and we denote the corresponding moduli space by
\[ \MCfibermod{\mathfrak{g}/F_{m+2}\mathfrak{g}}{x_m}, \]
which we call the \emph{moduli space of extensions} of $x_m$. In fact a simple calculation shows that
\begin{equation} \label{eqn_liftaction}
\exp(y)\cdot x_{m+1} = x_{m+1} + \{y,x_0\}
\end{equation}
for $x_{m+1}\in\MCfiberset{\mathfrak{g}/F_{m+2}\mathfrak{g}}{x_m}$ and $y\in F_{m+1}\mathfrak{g}^1/F_{m+2}\mathfrak{g}^1$.

\begin{prop} \label{prop_freeactonstructures}
The cohomology group $H^0\left(F_{m+1}\mathfrak{g}/F_{m+2}\mathfrak{g},\{x_0,-\}\right)$ acts freely and transitively on the moduli space $\MCfibermod{\mathfrak{g}/F_{m+2}\mathfrak{g}}{x_m}$ of extensions of $x_m$ by
\[ \xi_{m+1}\cdot x_{m+1} = x_{m+1} + \xi_{m+1}. \]
\end{prop}

\begin{proof}
This follows from Equation \eqref{eqn_liftQME} and Equation~\eqref{eqn_liftaction}.
\end{proof}

\subsubsection{Examples of associated cohomology theory}

\begin{example}
For $\ComBV{V}$ with its strong filtration by powers of $\hbar$
and a cyclic $\li$-structure $l$, which solves the classical master equation,
we may use \eqref{eqn_CEbracketdiff} to identify the cohomology theory associated to the corresponding level 0 structure $x_0$ as the Chevalley-Eilenberg cohomology of the $\li$-algebra,
\begin{equation} \label{eqn_CEfiltrationcohomology}
\left(F_m\ComBV{V}/F_{m+1}\ComBV{V},\{x_0,-\}\right) = \left(\csym V^*,-l\right).
\end{equation}
\end{example}

\begin{example}
For $\NonComBV{V}$ with its strong filtration by order and a cyclic $\ai$-structure $m$ on the space $A:=\Sigma^{-1}V$ given by a solution $x_0:=\tilde{m}$ to the classical master equation,
we may use \eqref{eqn_Hochbracketdiff} to identify the relevant cohomology theories that are associated to the level $0$ structure $x_0$ in terms of the cyclic cohomology of the $\ai$-algebra~$A$:
\begin{equation} \label{eqn_Hochfiltrationcohomology}
\left(F_m\NonComBV{V}/F_{m+1}\NonComBV{V},-\{x_0,-\}\right) = \prod_{\begin{subarray}{c} i,j,k\geq 0: \\ j+k>0 \\ 2i+j+k=m+1 \end{subarray}}\gamma^i\nu^j\left[\CycHoch{A}^{\cotimes k}\right]_{\symg{k}}.
\end{equation}
\end{example}

\begin{example}
Recall that any cyclic $\ai$-structure on $A$ determines a cyclic commutator $\li$-structure $l_N$ on $\gl{N}{A}$ that is $\gl{N}{\gf}$-invariant and which in turn yields a solution $x_0$
to the classical master equation for $\ComBV{\Sigma\gl{N}{A}}^{\gl{N}{\gf}}$. The cohomology theory associated to this solution is
\begin{equation} \label{eqn_invfiltrationcohomology}
\left(F_m\ComBV{\Sigma\gl{N}{A}}^{\gl{N}{\gf}}/F_{m+1}\ComBV{\Sigma\gl{N}{A}}^{\gl{N}{\gf}},\{x_0,-\}\right) = \left(\left[\csym\Sigma\gl{N}{A}^*\right]^{\gl{N}{\gf}},-l_N\right).
\end{equation}
Provided that the $\ai$-algebra $A$ is unital, it follows from Theorem \ref{thm_invariantqiso} that this complex computes the Chevalley-Eilenberg cohomology of~$\gl{N}{A}$.
\end{example}

\subsubsection{Finite level morphisms}

Having just described the obstruction theory that applies to producing solutions of the quantum master equation, we now repeat this process for morphisms between two quantizations of a solution $x_0$ to the classical master equation.

\begin{defi}
Given two level $m$ structures $x_m,x'_m\in\MCfiberset{\mathfrak{g}/F_{m+1}\mathfrak{g}}{x_0}$, a \emph{level $m$ morphism} from $x_m$ to $x'_m$ is an element
\[ y_m\in F_1\mathfrak{g}^1/F_{m+1}\mathfrak{g}^1 \]
such that
\[ x'_m = \exp(y_m)\cdot x_m. \]
We denote the set consisting of all such level $m$ morphisms by
\[ \Morphset{m}{x_m}{x'_m}. \]
Two level $m$ morphisms $y_m$, $y'_m$ from $x_m$ to $x'_m$ are said to be \emph{homotopy equivalent} if
\[ y_m = y'_m\star(d\eta +\{\eta,x_m\}) \]
for some $\eta\in F_1\mathfrak{g}^0/F_{m+1}\mathfrak{g}^0$. The set of homotopy equivalence classes of level $m$ morphisms will be denoted by
\[ \Morphmod{m}{x_m}{x'_m}. \]
\end{defi}

Suppose that $x_{m+1}$ and $x'_{m+1}$ are two level $(m+1)$ quantizations of $x_0$ that lift level $m$ quantizations $x_m$ and $x'_m$ respectively. The quotient map induces a natural map
\begin{equation} \label{eqn_leveldownmorphmap}
\Morphset{m+1}{x_{m+1}}{x'_{m+1}} \to \Morphset{m}{x_m}{x'_m}
\end{equation}
from level $(m+1)$ morphisms to level $m$ morphisms. As before, we wish to characterize the nonempty fibers of this map.

\begin{prop} \label{prop_obstructionmorph}
There is a well-defined obstruction map on the set of homotopy equivalence classes of level $m$ morphisms:
\begin{displaymath}
\begin{array}{ccc}
\Morphmod{m}{x_m}{x'_m} & \to & H^0\left(F_{m+1}\mathfrak{g}/F_{m+2}\mathfrak{g},\{x_0,-\}\right), \\
y_m & \mapsto & \Obs{y_m} := x'_{m+1} - \exp(y_m)\cdot x_{m+1}.
\end{array}
\end{displaymath}
A level $m$ morphism $y_m\in\Morphset{m}{x_m}{x'_m}$ may be lifted to a level $(m+1)$ morphism under the map \eqref{eqn_leveldownmorphmap} if and only if the cohomology class $\Obs{y_m}$ vanishes.
\end{prop}

\begin{proof}
Choosing a representative $\tilde{y}_m\in F_1\mathfrak{g}^1/F_{m+2}\mathfrak{g}^1$ for $y_m\in F_1\mathfrak{g}^1/F_{m+1}\mathfrak{g}^1$, we have tautologically that
\[ x'_{m+1} - \exp(\tilde{y}_m)\cdot x_{m+1}\in F_{m+1}\mathfrak{g}/F_{m+2}\mathfrak{g}. \]
Since this is the difference between two extensions of $x'_m$, it follows from Proposition \ref{prop_freeactonstructures} that this is a $\{x_0,-\}$-cocycle. Now if
\[ \tilde{y}'_m = \tilde{y}_m+\xi_{m+1},\quad\xi_{m+1}\in F_{m+1}\mathfrak{g}^1/F_{m+2}\mathfrak{g}^1; \]
is another choice of representative for $y_m$ then using \eqref{eqn_BVfiltration},
\begin{equation} \label{eqn_morphliftobs}
x'_{m+1} - \exp(\tilde{y}'_m)\cdot x_{m+1} = x'_{m+1} - \exp(\tilde{y}_m)\cdot x_{m+1} + \{x_0,\xi_{m+1}\}.
\end{equation}
Hence the expression for $\Obs{y_m}$ differs only by a coboundary. It is trivial to observe that this expression does not change under a homotopy equivalence.

Now it follows from Equation \eqref{eqn_morphliftobs} that the level $m$ morphism $y_m$ lifts to a level $(m+1)$ morphism $y_{m+1}=\tilde{y}_m+\xi_{m+1}$ if and only if the cohomology class $\Obs{y_m}$ vanishes.
\end{proof}

Given a level $m$ morphism $y_m\in\Morphset{m}{x_m}{x'_m}$, denote the fiber of the map \eqref{eqn_leveldownmorphmap} over the point $y_m$ by
\[ \Morphfiberset{m+1}{x_{m+1}}{x'_{m+1}}{y_m}. \]
Define two extensions $y_{m+1}$ and $y'_{m+1}$ in this fiber to be homotopy equivalent if for some $\eta\in F_{m+1}\mathfrak{g}^0/F_{m+2}\mathfrak{g}^0$, we have
\begin{equation} \label{eqn_morphliftact}
\begin{split}
y'_{m+1} &= y_{m+1}\star\left(d\eta+\{\eta,x_{m+1}\}\right)\\
&= y_{m+1} + \{x_0,\eta\}.
\end{split}
\end{equation}
Denote the set of homotopy equivalence classes of the fiber by
\[ \Morphfibermod{m+1}{x_{m+1}}{x'_{m+1}}{y_m}. \]

\begin{prop} \label{prop_freeactonmorphisms}
The cohomology group $H^1\left(F_{m+1}\mathfrak{g}/F_{m+2}\mathfrak{g},\{x_0,-\}\right)$ acts freely and transitively on the set $\Morphfibermod{m+1}{x_{m+1}}{x'_{m+1}}{y_m}$ of homotopy equivalence classes of extensions of $y_m$ by
\[ \xi_{m+1}\cdot y_{m+1} = y_{m+1} + \xi_{m+1}. \]
\end{prop}

\begin{proof}
This follows from Equation \eqref{eqn_morphliftobs} and Equation \eqref{eqn_morphliftact}.
\end{proof}

\subsection{Correspondence of moduli spaces}

In this section we use the results of the preceding subsection to establish sufficient criteria for a map between two strongly filtered differential graded Lie algebras to produce a one-to-one correspondence between the corresponding moduli spaces of quantizations.

Suppose that $\Phi:\mathfrak{g}\to\mathfrak{h}$
is a filtration-respecting map of strongly filtered differential graded Lie algebras. Given a solution $x_0\in\MCset{\mathfrak{g}/F_1\mathfrak{g}}$ to the classical master equation for $\mathfrak{g}$, we have a solution
\[ x'_0:=\Phi(x_0)\in\MCset{\mathfrak{h}/F_1\mathfrak{h}} \]
to the classical master equation for $\mathfrak{h}$. The map $\Phi$ induces a map
\begin{equation} \label{eqn_modspcmap}
\Phi:\MCfibermod{\mathfrak{g}}{x_0}\to\MCfibermod{\mathfrak{h}}{x'_0}
\end{equation}
between the moduli spaces of quantizations of these solutions.

\begin{theorem} \label{thm_modspccorrespondence}
If the map $\Phi:\mathfrak{g}\to\mathfrak{h}$ induces isomorphisms
\[ \Phi: H^{\bullet}\left(F_m\mathfrak{g}/F_{m+1}\mathfrak{g},\{x_0,-\}\right) \longrightarrow H^{\bullet}\left(F_m\mathfrak{h}/F_{m+1}\mathfrak{h},\{\Phi(x_0),-\}\right) \]
for all $m\geq 1$, then \eqref{eqn_modspcmap} is a bijection.
\end{theorem}

\begin{proof}
The method of proof follows a standard technique, cf. Theorem 2.4 of \cite{goldmill}, which we demonstrate here for the sake of completeness. The first step is to prove that \eqref{eqn_modspcmap} is surjective.

Given a quantization $x'\in\MCfiberset{\mathfrak{h}}{x'_0}$, denote the corresponding finite level structures by
\[ x'_m\in\MCfiberset{\mathfrak{h}/F_{m+1}\mathfrak{h}}{x'_0}, \quad m\geq 0. \]
We construct, by induction on $m$, sequences of finite level structures and morphisms:
\begin{equation} \label{eqn_constructsequence}
\begin{array}{ll}
x_m \in\MCfiberset{\mathfrak{g}/F_{m+1}\mathfrak{g}}{x_0}, & m\geq 0; \\
y_m \in\Morphset{m}{\Phi(x_m)}{x'_m}, & m\geq 0;
\end{array}
\end{equation}
such that each $x_{m+1}$ and $y_{m+1}$ is an extension of $x_m$ and $y_m$ respectively. Since the filtrations on $\mathfrak{g}$ and $\mathfrak{h}$ are complete Hausdorff filtrations, these finite level structures and morphisms will lift to elements
\begin{displaymath}
\begin{split}
x & \in\MCfiberset{\mathfrak{g}}{x_0}, \\
y & \in\Morph{\Phi(x)}{x'}.
\end{split}
\end{displaymath}
This will show that \eqref{eqn_modspcmap} is surjective.

Assume that the structures \eqref{eqn_constructsequence} have been constructed up to level $m$. Since $x'_m$ lifts to a level $(m+1)$ structure $x'_{m+1}$, it follows from Proposition \ref{prop_obstructionstruct} that
\[ 0=\Obs{x'_m}=\Obs{\Phi(x_m)}=\Phi\left(\Obs{x_m}\right). \]
Since $\Phi$ is injective we have $\Obs{x_m}=0$, so $x_m$ must lift to a level $(m+1)$ structure $\bar{x}_{m+1}$.

Choose an element $\bar{y}_{m+1}\in F_1\mathfrak{h}^1/F_{m+2}\mathfrak{h}^1$ that lifts the element $y_m\in F_1\mathfrak{h}^1/F_{m+1}\mathfrak{h}^1$. Then,
\[ \Phi(\bar{x}_{m+1}) \qquad\text{and}\qquad \exp(-\bar{y}_{m+1})\cdot x'_{m+1}, \]
are both level $(m+1)$ structures extending $\Phi(x_m)$. Therefore by Proposition \ref{prop_freeactonstructures} there exists a $\{x'_0,-\}$-cocycle $\xi'_{m+1}\in F_{m+1}\mathfrak{h}^0/F_{m+2}\mathfrak{h}^0$ such that we have equivalent extensions,
\[ \Phi(\bar{x}_{m+1}) + \xi'_{m+1} \approx \exp(-\bar{y}_{m+1})\cdot x'_{m+1}. \]
Since $\Phi$ is surjective, there is a $\{x_0,-\}$-cocycle $\xi_{m+1}\in F_{m+1}\mathfrak{g}^0/F_{m+2}\mathfrak{g}^0$ such that
\[ \Phi(\bar{x}_{m+1}+\xi_{m+1}) = \Phi(\bar{x}_{m+1}) +\Phi(\xi_{m+1}) \approx \exp(-\bar{y}_{m+1})\cdot x'_{m+1}. \]
Hence there is a $\beta_{m+1}\in F_{m+1}\mathfrak{h}^1/F_{m+2}\mathfrak{h}^1$ such that
\begin{displaymath}
\begin{split}
\exp(\beta_{m+1})\cdot\Phi(\bar{x}_{m+1}+\xi_{m+1}) &= \exp(-\bar{y}_{m+1})\cdot x'_{m+1} \\
\exp(\bar{y}_{m+1}\star\beta_{m+1})\cdot\Phi(\bar{x}_{m+1}+\xi_{m+1}) &= x'_{m+1}.
\end{split}
\end{displaymath}
The inductive step is completed by setting
\begin{displaymath}
\begin{split}
x_{m+1} &:= \bar{x}_{m+1}+\xi_{m+1}, \\
y_{m+1} &:= \bar{y}_{m+1}\star\beta_{m+1} = \bar{y}_{m+1}+\beta_{m+1}.
\end{split}
\end{displaymath}

To prove that \eqref{eqn_modspcmap} is injective, consider quantizations $x,x'\in\MCfiberset{\mathfrak{g}}{x_0}$. The map $\Phi$ induces a map
\[ \Phi:\MorphH{x}{x'}\longrightarrow\MorphH{\Phi(x)}{\Phi(x')} \]
on homotopy equivalence classes of morphisms. It suffices to prove that this map is surjective, which is accomplished in exactly the same fashion as before using Proposition \ref{prop_obstructionmorph} and Proposition \ref{prop_freeactonmorphisms}.
\end{proof}

\subsection{Examples of Theorem \ref{thm_modspccorrespondence}}

We consider here some applications of Theorem \ref{thm_modspccorrespondence} to our running examples.

Let $A$ be a \emph{unital} cyclic $\ai$-algebra whose inner product has odd degree, and consider the map
\[ \widehat{\Mor}_{\gamma,\nu}:\NonComBV{\Sigma A}\longrightarrow\NonComBV{\Sigma\mat{N}{A}} \]
given by \eqref{eqn_OTFTmatrixdglamap}. This map is a map of strongly filtered differential graded Lie algebras, where the filtrations on each side are those defined by Example \ref{exm_noncomBVfiltration}. The cyclic $\ai$-structure $m_A$ on $A$ yields a solution $x_0:=\tilde{m}_A$ to the classical master equation for $\NonComBV{\Sigma A}$. The corresponding cyclic $\ai$-structure $m_{\mat{N}{A}}$ on matrices yields a solution $x'_0=\widehat{\Mor}_{\gamma,\nu}(x_0)$ to the classical master equation for $\NonComBV{\Sigma\mat{N}{A}}$.

\begin{theorem}
For $A$ a \emph{unital} cyclic $\ai$-algebra whose inner product has odd degree,
the map $\widehat{\Mor}_{\gamma,\nu}$ induces a one-to-one correspondence
\[ \widehat{\Mor}_{\gamma,\nu}:\MCfibermod{\NonComBV{\Sigma A}}{x_0}\longrightarrow\MCfibermod{\NonComBV{\Sigma\mat{N}{A}}}{x'_0} \]
between these moduli spaces of quantizations.
\end{theorem}

\begin{proof}
It follows from Equation \eqref{eqn_Hochfiltrationcohomology}, Theorem \ref{thm_OTFTmatrixdglamap} and Theorem \ref{thm_moritainvariance} that $\widehat{\Mor}_{\gamma,\nu}$ satisfies the hypothesis of Theorem \ref{thm_modspccorrespondence}.
\end{proof}

For another application of this theorem, recall from Example \ref{exm_comglnfiltration} that the cyclic commutator $\li$-structure $l_N$ on $\gl{N}{A}$ determined by the unital cyclic $\ai$-structure on $A$ yields a solution $x_0$ to the classical master equation for $\ComBV{\Sigma\gl{N}{A}}^{\gl{N}{\gf}}$ and likewise for $\ComBV{\Sigma\gl{N}{A}}$. Consider the inclusion,
\begin{equation} \label{eqn_inclusionofinvariants}
\ComBV{\Sigma\gl{N}{A}}^{\gl{N}{\gf}}\longrightarrow\ComBV{\Sigma\gl{N}{A}}
\end{equation}
of the differential graded Lie subalgebra of invariants.

\begin{theorem}
For $A$ a \emph{unital} cyclic $\ai$-algebra whose inner product has odd degree,
the map \eqref{eqn_inclusionofinvariants} induces a one-to-one correspondence,
\[ \MCfibermod{\ComBV{\Sigma\gl{N}{A}}^{\gl{N}{\gf}}}{x_0}\longrightarrow\MCfibermod{\ComBV{\Sigma\gl{N}{A}}}{x_0} \]
between these moduli spaces of quantizations.
\end{theorem}

\begin{proof}
It follows from Equation \eqref{eqn_CEfiltrationcohomology}, Equation \eqref{eqn_invfiltrationcohomology} and Theorem \ref{thm_invariantqiso} that \eqref{eqn_inclusionofinvariants} satisfies the hypothesis of Theorem \ref{thm_modspccorrespondence}.
\end{proof}

\begin{rem}
This theorem tells us that a quantization of our cyclic commutator $\li$-structure $l_N$ is essentially the same thing as a \emph{conjugation invariant} quantization of our $\li$-structure; or more precisely, that any quantization of $l_N$ may be replaced by an equivalent conjugation invariant quantization that is unique up to gauge equivalence. For this reason, in the following section we will focus only on such \emph{invariant} quantizations. The fundamental theorem of invariant theory for $\gl{N}{\gf}$ (upon which the proof of the Loday-Quillen-Tsygan Theorem is based) tells us essentially what such invariant quantizations of the action must look like; they must be generated by the trace maps \eqref{eqn_traceproduct} on matrices.
\end{rem}

\section{Noncommutative geometry in the large $N$ limit} \label{sec_largeN}

In this section we explain how the noncommutative geometry of Section \ref{sec_BVNonComGeom} and \ref{sec_MinvLQTBV} emerges naturally by considering the problem of the Batalin-Vilkovisky quantization of Chern-Simons type gauge theories in the large $N$ limit. In order to examine the large $N$ behavior of observables within the Batalin-Vilkovisky formalism, it is necessary to define a quantization of your action at \emph{each} rank $N$ of your theory. This leads us directly to the problem of \emph{simultaneous quantization}; we must carry out the quantization process so that it works simultaneously at every rank $N$. It is this problem that the differential graded Lie algebra constructed in Definition \ref{def_noncomquantumBV} solves, and it is the Loday-Quillen-Tsygan Theorem that explains why. The family \eqref{eqn_LQTdglamaps} of maps from this object that target each rank $N$ of the theory clearly provides a mechanism for any quantization that is built in this object to yield a quantization at all ranks $N$. In this way, this differential graded Lie algebra plays the role of a universal object, and any quantization built in it plays the role of a quantization that is `universal in $N$.' One of the main purposes of this section is to explain to what extent this construction is \emph{necessary}, and we prove what we consider to be a reasonable converse; the obstructions to quantization vanish at all ranks $N$ of the theory if and only if the obstruction to quantization vanishes in this universal object built using noncommutative geometry.

Let us explain in more detail. Our starting point is a cyclic $\ai$-algebra $A$ whose inner product has odd degree. We emphasize that for the purposes of this section, this algebra does \emph{not} need to be unital and that all the results of this section hold for nonunital algebras. From the viewpoint of Chern-Simons theory our motivating example would be to take $A$ to be the de Rham algebra on a closed odd-dimensional manifold; however, this of course is not finite-dimensional and so does not comport with Definition \ref{def_ainfinitycyclic}. Nonetheless, to keep this motivating point of view alive we may imagine that a finite-dimensional $A$ has been obtained from such an example through the homotopy transfer theorem. Denote the solution to the classical master equation corresponding to our cyclic $\ai$-structure by
\[\tilde{m}\in\MCset{\NCHam{\Sigma A}}, \]
and likewise denote by
\[\tilde{l}_N\in\MCset{\left[\csym\Sigma\gl{N}{A}^*\right]^{\gl{N}{\gf}}}, \quad N\geq 1, \]
the solutions to the classical master equations that are formed by the corresponding commutator $\li$-structures, as in Example~\ref{exm_comglnfiltration}.

Consider the diagram of strongly filtered differential graded Lie algebras (with, we emphasize, no horizontal arrows)
\begin{equation} \label{eqn_NCtoCS}
\xymatrix{ & \NonComBV{\Sigma A} \ar[ld]^{\qquad\cdots\cdots} \ar[d] \ar[dr] \ar[drrr]_{\cdots\cdots\quad} \\ \ComBV{\Sigma\gl{1}{A}}^{\gl{1}{\gf}} \ar@{.}[r] & \ComBV{\Sigma\gl{N}{A}}^{\gl{N}{\gf}} & \mkern-18mu \ComBV{\Sigma\gl{N+1}{A}}^{\gl{N+1}{\gf}} \ar@{.}[r] && }
\end{equation}
formed from the family of maps \eqref{eqn_LQTdglamaps}.
Applying the Maurer-Cartan functor to this yields a diagram of moduli spaces of quantizations:
\[ \xymatrix{ & \MCfibermod{\NonComBV{\Sigma A}}{\tilde{m}} \ar[d]_-{\cdots\cdots} \ar[rd]^{\cdots\cdots} \\ \ar@{.}[r] & \MCfibermod{\ComBV{\Sigma\gl{N}{A}}^{\gl{N}{\gf}}}{\tilde{l}_N} & \MCfibermod{\ComBV{\Sigma\gl{N+1}{A}}^{\gl{N+1}{\gf}}}{\tilde{l}_{N+1}} \ar@{.}[r] & } \]
It follows from this picture that any quantization of the cyclic $\ai$-structure $\tilde{m}$ living in $\NonComBV{\Sigma A}$ yields conjugation invariant quantizations of the corresponding Chern-Simons type theories \emph{for each rank $N$ of the theory}. These quantizations are formed by the images of the quantization of $\tilde{m}$ under the above family of mappings.

\begin{rem}
Consider the horizontal mappings
\[ \ComBV{\Sigma\gl{N}{A}}^{\gl{N}{\gf}} \longleftarrow \ComBV{\Sigma\gl{N+1}{A}}^{\gl{N+1}{\gf}} \]
that arise from the inclusion of $\gl{N}{A}$ into $\gl{N+1}{A}$.
These are \emph{not} maps of differential graded Lie algebras. Consequently, there is no natural way to pass from a quantization at some rank $N$ of the Chern-Simons type theory to a quantization at some lower rank. This emphasizes the utility of noncommutative geometry in the above in being able to produce quantizations at all ranks $N$.

Furthermore, we should mention that these horizontal maps do \emph{not} make the above diagram commute and hence the rank $N$ quantizations that are produced from a quantization of the cyclic $\ai$-structure $\tilde{m}$ are \emph{not} related by these mappings. The reason for this is the peculiar behavior of the parameter $\nu$ under the map \eqref{eqn_OTFTmatrixdglamap} which is sensitive to the rank $N$ of the theory.
\end{rem}

The purpose of the rest of this section is to describe a partial converse to this result. Consider the problem of \emph{simultaneously} quantizing the Chern-Simons type actions $\tilde{l}_N$ at each rank $N$ of the theory. As we know from Section \ref{sec_obsthy}, this will require the corresponding family of obstructions to vanish. We will see in Theorem \ref{thm_simquantaction} that, as a consequence of the Loday-Quillen-Tsygan Theorem, this will happen if and only if the obstruction to quantizing the action $\tilde{m}$ vanishes. In this sense the problem of simultaneously quantizing the Chern-Simons type actions at every rank $N$ leads us to build our quantization in the differential graded Lie algebra $\NonComBV{\Sigma A}$.

We require the following lemma, which follows as a simple consequence of the Loday-Quillen-Tsygan Theorem.

\begin{lemma} \label{lem_kerintzero}
Consider, for any fixed $p\geq 1$, the family of maps
\begin{multline} \label{eqn_uniquantmaps}
H^{\bullet}(F_p\NonComBV{\Sigma A}/F_{p+1}\NonComBV{\Sigma A},\{\tilde{m},-\}) \longrightarrow \\ H^{\bullet}(F_p\ComBV{\Sigma\gl{N}{A}}^{\gl{N}{\gf}}/F_{p+1}\ComBV{\Sigma\gl{N}{A}}^{\gl{N}{\gf}},\{\tilde{l}_N,-\}), \qquad N\geq 1;
\end{multline}
defined for a cyclic $\ai$-algebra $A$ whose inner product has odd degree, which are induced by the maps \eqref{eqn_NCtoCS} of strongly filtered differential graded Lie algebras. The intersection of the kernels of this family of maps is the zero subspace.
\end{lemma}

\begin{proof}
First, using \eqref{eqn_Hochfiltrationcohomology} and \eqref{eqn_invfiltrationcohomology} we may write;
\begin{multline} \label{eqn_Hochsplit}
H^{\bullet}(F_p\NonComBV{\Sigma A}/F_{p+1}\NonComBV{\Sigma A},\{\tilde{m},-\}) = \prod_{\begin{subarray}{c} i,j,k\geq 0: \\ j+k>0 \\ 2i+j+k=p+1 \end{subarray}}\gamma^i\nu^j H^{\bullet}\left(\left[\CycHoch{A}^{\cotimes k}\right]_{\symg{k}}\right) \\
=\left(\prod_{\begin{subarray}{c} i\geq 0,j\geq 1: \\ 2i+j=p+1 \end{subarray}}\gamma^i\nu^j\gf\right)\times\left(\prod_{\begin{subarray}{c} i,j\geq 0,k\geq 1: \\ 2i+j+k=p+1 \end{subarray}} \gamma^i\nu^j H^{\bullet}\left(\left[\CycHoch{A}^{\cotimes k}\right]_{\symg{k}}\right)\right)
\end{multline}
and
\begin{multline} \label{eqn_CEsplit}
H^{\bullet}(F_p\ComBV{\Sigma\gl{N}{A}}^{\gl{N}{\gf}}/F_{p+1}\ComBV{\Sigma\gl{N}{A}}^{\gl{N}{\gf}},\{\tilde{l}_N,-\}) = \\
H^{\bullet}\left(\left[\csym(\Sigma\gl{N}{A}^*)\right]^{\gl{N}{\gf}},l_N\right) = \gf\times H^\bullet\left(\left[\ChEilp{\gl{N}{A}}\right]^{\gl{N}{\gf}}\right).
\end{multline}
The maps \eqref{eqn_uniquantmaps} respect the decompositions \eqref{eqn_Hochsplit} and \eqref{eqn_CEsplit}. Therefore it makes sense to begin by examining the behavior of these maps on the left-hand factors:
\[ \prod_{\begin{subarray}{c} i\geq 0,j\geq 1: \\ 2i+j=p+1 \end{subarray}}\gamma^i\nu^j\gf \longrightarrow \gf, \quad \gamma^i\nu^j \mapsto N^j; \qquad N\geq 1. \]
Since we work in characteristic zero, and since a nontrivial polynomial may have only finitely many roots, it follows that the intersection of the kernels of the above maps is trivial.

Now we examine the behavior of the maps \eqref{eqn_uniquantmaps} on the right-hand factors of \eqref{eqn_Hochsplit} and \eqref{eqn_CEsplit}, which we rewrite as follows:
\begin{equation} \label{eqn_uniquantmapsdiag}
\xymatrix{\prod_{\begin{subarray}{c} i,j\geq 0,k\geq 1: \\ 2i+j+k=p+1 \end{subarray}} \gamma^i\nu^j H^{\bullet}\left(\left[\CycHoch{A}^{\cotimes k}\right]_{\symg{k}}\right) \ar[r] \ar@{=}[d] & H^\bullet\left(\left[\ChEilp{\gl{N}{A}}\right]^{\gl{N}{\gf}}\right) \\ \prod_{j=0}^p\nu^j\left(\prod_{i=0}^{\left[\frac{p-j}{2}\right]}\gamma^i H^{\bullet}\left(\left[\CycHoch{A}^{\cotimes p+1-2i-j}\right]_{\symg{p+1-2i-j}}\right)\right) \ar[ru] }
\end{equation}
where the top row is the map \eqref{eqn_uniquantmaps} and the diagonal map may be described as follows. Denote the natural projection maps of the inverse limit \eqref{eqn_LQTinvlimhomology} by
\[ \pi_N:\prod_{k=1}^\infty H^{\bullet}\left(\left[\CycHoch{A}^{\cotimes k}\right]_{\symg{k}}\right) = H^{\bullet}\left(\csymp\left(\CycHoch{A}\right)\right) \longrightarrow H^\bullet\left(\left[\ChEilp{\gl{N}{A}}\right]^{\gl{N}{\gf}}\right). \]
Then the diagonal map of \eqref{eqn_uniquantmapsdiag} is the map
\[ \sum_{j=0}^p \nu^j z_j \mapsto \sum_{j=0}^p N^j\pi_N(z_j), \qquad z_j\in\prod_{i=0}^{\left[\frac{p-j}{2}\right]}\gamma^i H^{\bullet}\left(\left[\CycHoch{A}^{\cotimes p+1-2i-j}\right]_{\symg{p+1-2i-j}}\right), \]
where $\gamma^i H^{\bullet}\left(\left[\CycHoch{A}^{\cotimes k}\right]_{\symg{k}}\right)$ is identified with $H^{\bullet}\left(\left[\CycHoch{A}^{\cotimes k}\right]_{\symg{k}}\right)$ by setting $\gamma:=1$.

Now suppose that the element $z=\sum_{j=0}^p \nu^j z_j$ lies in the kernel of every map \eqref{eqn_uniquantmaps}, then for all $M\geq N\geq 1$,
\[ \sum_{j=0}^p M^j\pi_M(z_j) = 0 \qquad\text{and therefore}\qquad \sum_{j=0}^p M^j\pi_N(z_j) = 0; \]
where in the last equation we have applied the horizontal arrows of the inverse limit \eqref{eqn_LQTinvlim}. Just as before, since we work in characteristic zero, it follows that
\[ \pi_N(z_j) = 0, \qquad\text{for all $N\geq 1$ and $j=0,\ldots,p$.} \]
From \eqref{eqn_LQTinvlimhomology} it follows that each $z_j=0$.
\end{proof}

Now we may formulate and prove our theorem on quantization in the large $N$ limit.

\begin{theorem} \label{thm_simquantaction}
Let $A$ be a cyclic $\ai$-algebra whose inner product has odd degree and let
\[ x_k\in\MCfibermod{\NonComBV{\Sigma A}/F_{k+1}\NonComBV{\Sigma A}}{\tilde{m}} \]
be a level $k$ quantization of the cyclic $\ai$-structure $\tilde{m}$. Consider the corresponding family
\[ w_k^N\in\MCfibermod{\ComBV{\Sigma\gl{N}{A}}^{\gl{N}{\gf}}/F_{k+1}\ComBV{\Sigma\gl{N}{A}}^{\gl{N}{\gf}}}{\tilde{l}_N}, \quad N\geq 1, \]
of level $k$ quantizations of the cyclic $\li$-structures $\tilde{l}_N$ that is determined by the diagram \eqref{eqn_NCtoCS} of strongly filtered differential graded Lie algebras.

The level $k$ quantization $x_k$ of $\tilde{m}$ extends to a level $(k+1)$ quantization if and only if every level $k$ quantization $w^N_k$ of $\tilde{l}_N$ extends to a level $(k+1)$ quantization for all $N\geq 1$.
\end{theorem}

\begin{proof}
Obviously any level $(k+1)$ extension $x_{k+1}$ of the level $k$ quantization $x_k$ yields level $(k+1)$ extensions of the level $k$ quantizations $w_k^N$ for all $N\geq 1$ by considering the images of $x_{k+1}$ under the maps \eqref{eqn_NCtoCS}.

Conversely suppose that each level $k$ quantization $w^N_k$ extends to a level $(k+1)$ quantization for every $N\geq 1$. By Proposition \eqref{prop_obstructionstruct} the family of obstructions
\[ \Obs{w^N_k}\in H^{\bullet}\left(F_{k+1}\ComBV{\Sigma\gl{N}{A}}^{\gl{N}{\gf}}/F_{k+2}\ComBV{\Sigma\gl{N}{A}}^{\gl{N}{\gf}},\{\tilde{l}_N,-\}\right), \qquad N\geq 1; \]
must all vanish. Since these obstructions are the images of the obstruction $\Obs{x_k}$ under the family of maps \eqref{eqn_uniquantmaps} it follows from Lemma \ref{lem_kerintzero} that the obstruction $\Obs{x_k}$ must vanish and hence by Proposition \ref{prop_obstructionstruct} that the level $k$ quantization $x_k$ may be extended to a level $(k+1)$ quantization.
\end{proof}

\begin{rem}
To reiterate, let us lay out once more how the above theorem explains the emergence of noncommutative geometry in the large $N$ limit quantization of Chern-Simons type theories. We begin with a solution $x_0$ to the classical master equation and in order to quantize \emph{every} Chern-Simons type action $w^N_0$ we must, by the above theorem, be able to extend $x_0$ to a level 1 quantization $x_1$. This solution $x_1$ leads to level 1 quantizations $w^N_1$  for all $N\geq 1$ by virtue of Diagram \eqref{eqn_NCtoCS} and in order to extend \emph{all} of these we must, again by the above theorem, be able to extend $x_1$ to a level 2 quantization $x_2$... and so the process goes on with us building our quantization $x$ in $\NonComBV{\Sigma A}$.
\end{rem}

We can also formulate an analogue of Theorem \ref{thm_simquantaction} for morphisms, which is naturally more technical.

\begin{theorem}
Let $A$ be a cyclic $\ai$-algebra whose inner product has odd degree and let
\[ x_k,x'_k\in\MCfiberset{\NonComBV{\Sigma A}/F_{k+1}\NonComBV{\Sigma A}}{\tilde{m}} \]
be two level $k$ quantizations of the cyclic $\ai$-structure $\tilde{m}$. Suppose that
\[ x_{k+1},x'_{k+1}\in\MCfiberset{\NonComBV{\Sigma A}/F_{k+2}\NonComBV{\Sigma A}}{\tilde{m}} \]
are two level $(k+1)$ quantizations extending $x_k$ and $x'_k$ respectively. Suppose further that we are given a level $k$ morphism
\[ y_k\in\MorphH{x_k}{x'_k}. \]
Now consider the quantizations and morphisms
\begin{displaymath}
\begin{array}{c}
w^N_k,w'^N_k\in\MCfiberset{\ComBV{\Sigma\gl{N}{A}}^{\gl{N}{\gf}}/F_{k+1}\ComBV{\Sigma\gl{N}{A}}^{\gl{N}{\gf}}}{\tilde{l}_N} \\
w^N_{k+1},w'^N_{k+1}\in\MCfiberset{\ComBV{\Sigma\gl{N}{A}}^{\gl{N}{\gf}}/F_{k+2}\ComBV{\Sigma\gl{N}{A}}^{\gl{N}{\gf}}}{\tilde{l}_N} \\
Y^N_k\in\MorphH{w^N_k}{w'^N_k}
\end{array}
, \qquad N\geq 1;
\end{displaymath}
which are the respective images of the preceding quantizations and morphisms under the family of maps~\eqref{eqn_NCtoCS}.

The level $k$ morphism $y_k$ from $x_k$ to $x'_k$ extends to a level $(k+1)$ morphism from $x_{k+1}$ to $x'_{k+1}$ if and only if for all $N\geq 1$, the level $k$ morphism $Y^N_k$ from $w^N_k$ to $w'^N_k$ extends to a level $(k+1)$ morphism from $w^N_{k+1}$ to $w'^N_{k+1}$.
\end{theorem}

\begin{proof}
The proof follows the proof of Theorem \ref{thm_simquantaction} mutatis mutandis using Proposition \ref{prop_obstructionmorph} and Lemma \ref{lem_kerintzero}.
\end{proof}

\end{document}